\documentclass[12pt]{amsart}
\usepackage[utf8]{inputenc}

\usepackage{geometry}
\usepackage{setspace}
\usepackage{graphicx}
\usepackage{amssymb}
\usepackage{mathrsfs}
\usepackage{amsthm}
\usepackage{amsmath}
\usepackage{color}
\usepackage{caption}
\usepackage{comment}
\usepackage{subcaption}
\usepackage{mathtools}
\usepackage{enumerate}
\usepackage{algorithm,algorithmic}

\usepackage{hyperref}

\theoremstyle{plain} 
\newtheorem{theorem}{Theorem}[section]

\newtheorem{corollary}[theorem]{Corollary}
\newtheorem{proposition}[theorem]{Proposition}
\newtheorem{lemma}[theorem]{Lemma}

\theoremstyle{definition} 
\newtheorem{example}[theorem]{Example}
\newtheorem{assumption}[theorem]{Assumption}

\theoremstyle{remark} 
\newtheorem{remark}[theorem]{Remark}

\numberwithin{equation}{section} 

\makeatletter
\newcommand{\myitem}[1]{
\item[#1]\protected@edef\@currentlabel{#1}
}
\makeatother

\title[Surface finite element approximation of SPDEs with Whittle--Matérn noise]{Surface finite element approximation of parabolic SPDEs with Whittle--Matérn noise}

\author[Ø.~S.~Auestad]{Øyvind S. Auestad} 
\address[Øyvind S. Auestad]{\newline Department of Mathematical Sciences Norwegian University of Science and Technology, \newline 7034 Trondheim, Norway.} \email[(corresponding author)]{oyvinau@ntnu.no}

\author[G.-A.~Fuglstad]{Geir-Arne Fuglstad}
\address[Geir-Arne Fuglstad]{\newline Department of Mathematical Sciences Norwegian University of Science and Technology \newline NO--7491 Trondheim, Norway.} \email{geir-arne.fuglstad@ntnu.no}

\author[A.~Lang]{Annika Lang} \address[Annika Lang]{\newline Department of Mathematical Sciences Chalmers University of Technology and University of Gothenburg \newline S--412 96 G\"oteborg, Sweden.} \email{annika.lang@chalmers.se}

\begin{document}

\begin{abstract}
    We propose and analyse a new type of fully discrete surface finite element approximation of a class of linear parabolic stochastic evolution equations with additive noise. Our discretization uses a surface finite element approximation of the noise, and is tailored for equations with noise having covariance operator defined by (negative powers of) elliptic operators, like Whittle--Matérn random fields. We derive strong and pathwise convergence rates of our approximation, and verify these by numerical experiments.
\end{abstract}

\date{\today}

\maketitle

\section{Introduction}

We study linear parabolic stochastic evolution equations of the form
\begin{align}\label{eq:model-in-intro}
	du(t) = -A_1 u(t) \, dt + A_2^{-\gamma} \, dW(t), \quad u(0) = \xi, \quad t \in [0,T],
\end{align}
where $W$ is a cylindrical Wiener process on $L^2(\Gamma)$, $\Gamma \subseteq \mathbb{R}^{d+1}$ a compact hypersurface, $\gamma > d/4-1/2$ a parameter, $\xi$ a random initial condition, and $A_j$, $j = 1, 2$, second order elliptic operators of the form, 
\begin{align*}
	A_j = \alpha_j - \nabla_{\Gamma} \cdot \mathcal{A}_j \nabla_{\Gamma} + b_j \cdot \nabla_{\Gamma}, \quad j = 1, 2,
\end{align*}
with bounded and measurable coefficients---see Assumption \ref{assumption:model} for details. The model \eqref{eq:model-in-intro} has additive noise with Whittle--Matérn type covariance, and is popular in spatial statistics. In this paper we propose and analyse a fully discrete surface finite element approximation of \eqref{eq:model-in-intro}, which uses a backward Euler approximation in time and a quadrature approximation of the fractional power operator.

Numerical approximations of solutions to stochastic partial differential equations (SPDEs) have been studied mainly for equations posed on Euclidean domains over the last three decades. Important problems in the climate and environmental sciences and biology warrant the development and analysis of numerical approximations of SPDEs defined on domains like Earth or the surface of a cell. The literature on surface finite element approximations of parabolic SPDEs is scarce. On the other hand, spectral approximations of SPDEs on the sphere have been studied exhaustively (see, e.g., \cite{2015-lang, 2019-kazashi, 2022-lang, 2024-lang}). Further, spatial statisticians have used spectral approximations, ad-hoc finite element and finite volume approximations of models similar to \eqref{eq:model-in-intro} to define stochastic models on surfaces, and applied these to problems in the climate and environmental sciences (see, e.g., \cite{2011-rue, 2013-cameletti, 2020-fuglstad, 2022-lindgren, 2024-lindgren}).

Spectral approximations of \eqref{eq:model-in-intro} require knowledge of eigenvectors and eigenvalues of $A_j$, and therefore impose severe restrictions on $A_j$ and the domain $\Gamma$. Further, the models and numerical approximations applied in spatial statistics are often missing a rigorous formulation and convergence analysis. Thus, there is a gap between the available discretizations and convergence results in the numerical analysis literature, and the models that are actually being applied in statistics. The surface finite element approximation proposed in this paper is applicable for a large class of elliptic operators $A_j$, not necessarily selfadjoint, and arbitrary compact smooth hypersurfaces $\Gamma$. It is easy to implement, computationally efficient, and under standard conditions, one may verify strong and pathwise convergence rates identical to those in the analogous flat domain case (see \cite{2025-auestad}).

More precisely, for an approximate surface $\Gamma_h \approx \Gamma$ and a finite element space $V_h$ on $\Gamma_h$, our fully discrete surface finite element approximation, evaluated at $t_n := n \Delta t$ for some time step size $\Delta t > 0$ and integer $n$, is based on the sinc quadrature approximation of $A_2^{-\gamma}$ of \cite{2019-bonito}, and a backward Euler approximation in time. Its basis coefficients in the nodal basis at time $t_n$, denoted $\{ \alpha_j^n\}_{j = 1}^{N_h}$, are defined by
\begin{align}\label{eq:scheme-basis-coefficients}
\begin{split}
    &(M_h + \Delta t T_h) \alpha^{n+1} \\
    &\qquad =
    \begin{cases}
        M_h \alpha^n + M_h \Delta t^{1/2} \frac{k \sin(\pi \gamma)}{\pi} \sum_{j = -M}^N e^{(1-\gamma) j k} (e^{j k} M_h + K_h)^{-1} M_h^{1/2} \varrho_h^n, \quad &\gamma \in (0,1) \\
        M_h \alpha^n + M_h \Delta t^{1/2} K_h^{-1} M_h^{1/2} \varrho_h^n, \quad &\gamma = 1,
    \end{cases}
\end{split}
\end{align}
where the matrices $M_h, T_h$ and $K_h$ are the usual finite element matrices, given in \eqref{eq:fem-matrices}, $\varrho_h^n \sim \mathcal{N}(0,I)$ are $N_h$-dimensional Gaussian and independent, while $k, N, M$ are parameters related to the sinc quadrature approximation. Therefore, simulating from our fully discrete approximation only entails assembling the usual finite element matrices and computing a Cholesky factor of the mass matrix, $M_h$. In Figure \ref{fig:motivation} four realizations of $u(1)$ are displayed for different values of $\gamma$ and domains $\Gamma$ (see Example \ref{ex:motivation} for details).

The numerical approximation proposed in this paper is tailored for equations with noise having covariance operator defined by negative fractional powers of elliptic operators. Covariance operators of this form are known as Whittle--Matérn covariance operators, due to their connection to the Matérn covariance function. It may be shown that $(I - \Delta)^{-\gamma} \mathcal{W}$ has Matérn covariance function when $\mathcal{W}$ is white noise on $L^2(\mathbb{R}^d)$, and $\gamma > 0$ is large enough \cite{1954-whittle}. Thus, for each $t$ we may view the noise $A_2^{-\gamma} W (t)$ in \eqref{eq:model-in-intro} as a generalization of a Matérn random field. These random fields are popular models in spatial statistics \cite{2011-rue, 2012-stein}, and their numerical approximations have been studied on both Euclidean space and surfaces (see, e.g., \cite{2018-bolin, 2019-bolin}, \cite{2024-bonito, 2019-herrmann, 2023-lang}, and the references therein). In particular, our approximation of the additive noise is inspired by the surface finite element approximation proposed in \cite{2024-bonito} for this stationary problem.

\begin{figure}
    \centering
    \subcaptionbox{$\gamma = 0.1$}{
        \includegraphics[width = 0.22\textwidth]{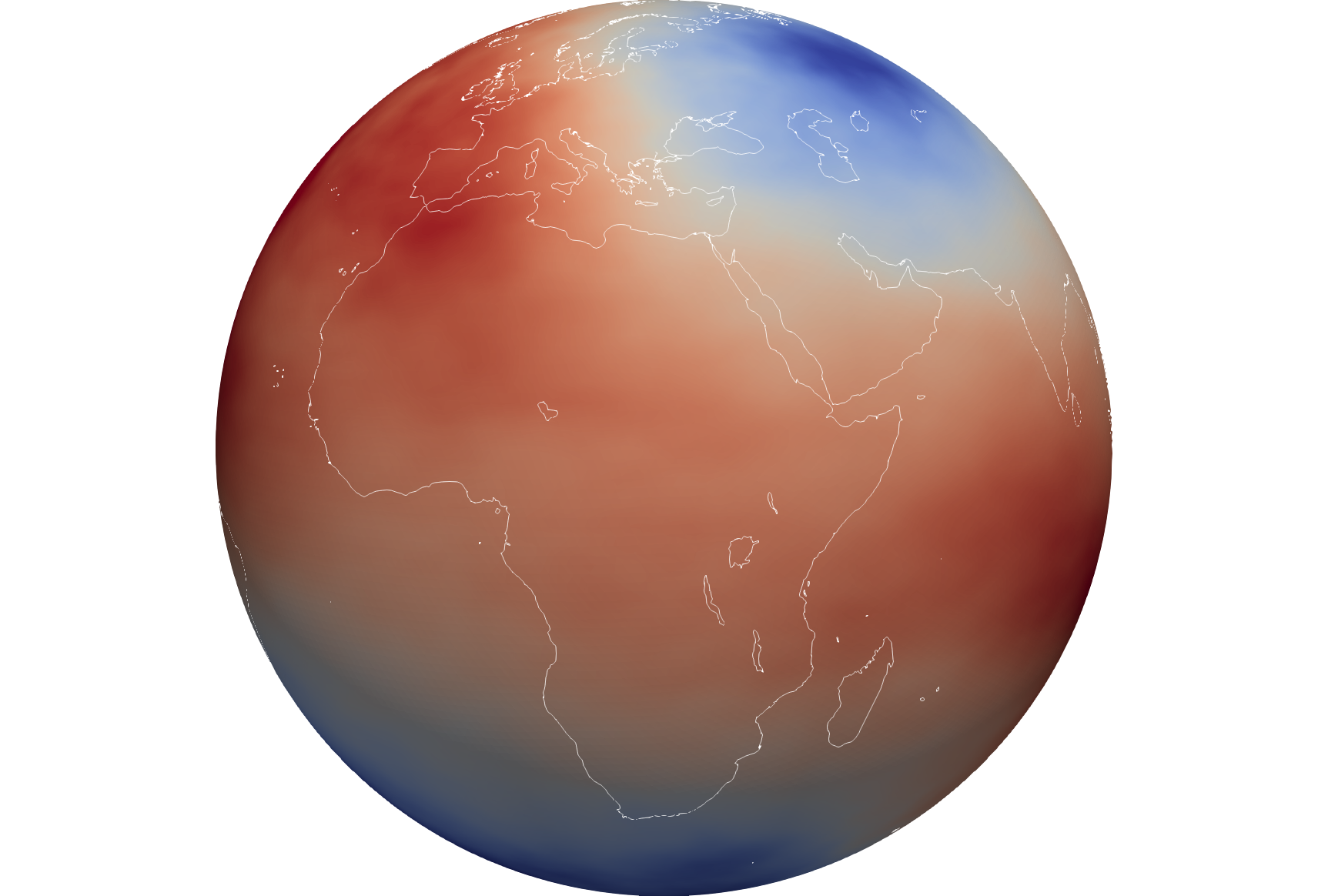}
    }
    \subcaptionbox{$\gamma = 0.9$}{
        \includegraphics[width = 0.22\textwidth]{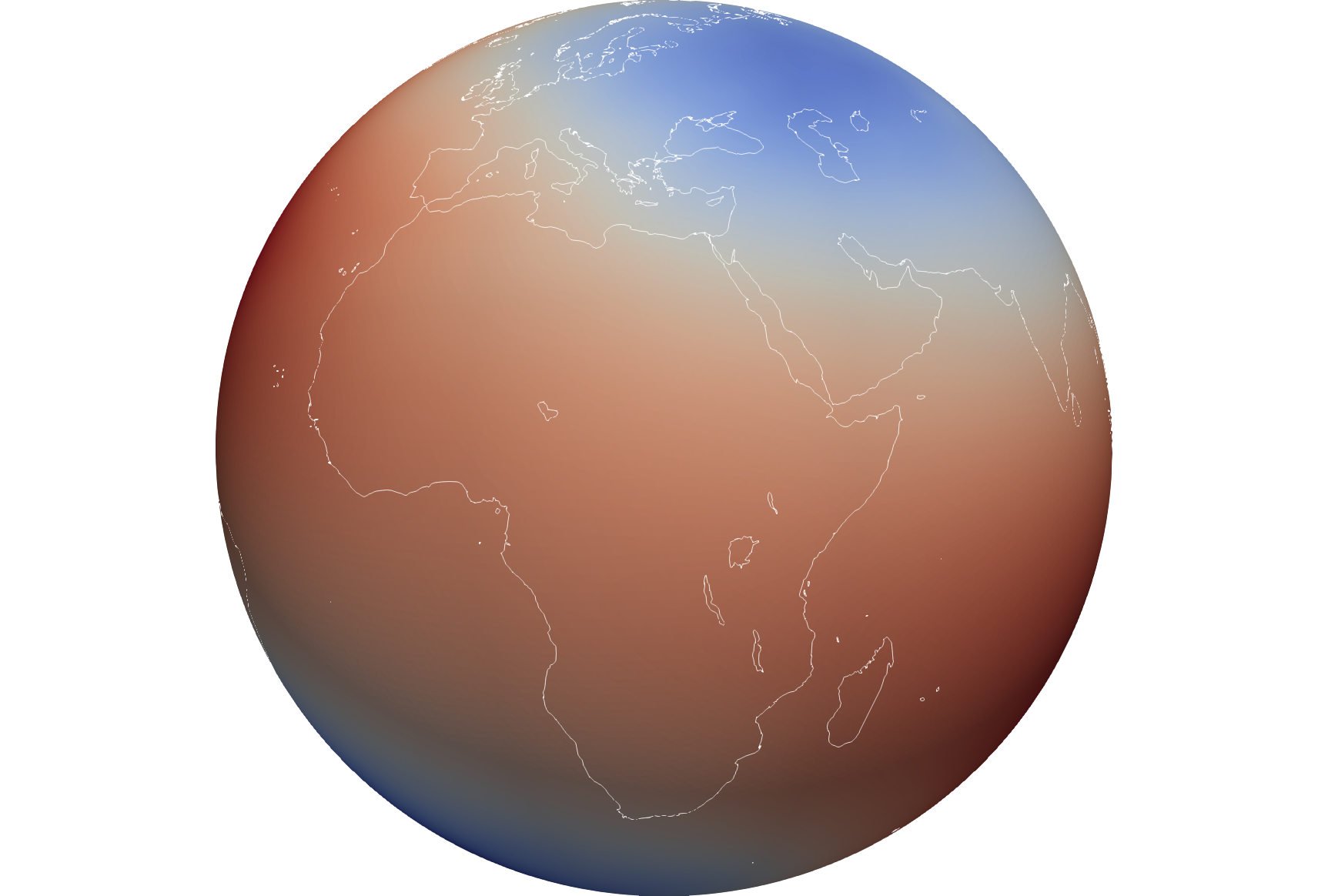}
    }
    \subcaptionbox{$\gamma = 0.1$}{
        \includegraphics[width = 0.22\textwidth]{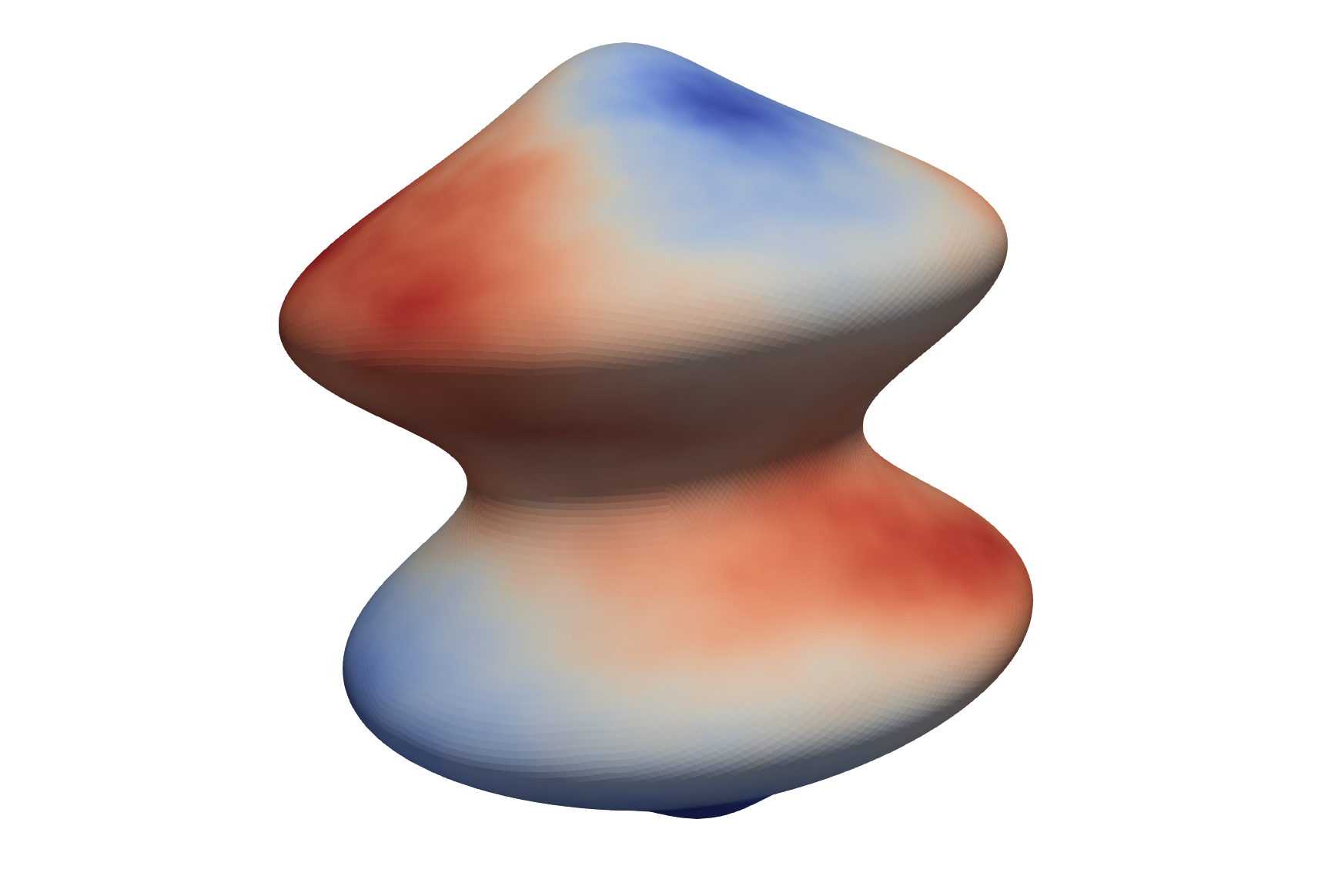}
    }
    \subcaptionbox{$\gamma = 0.9$}{
        \includegraphics[width = 0.22\textwidth]{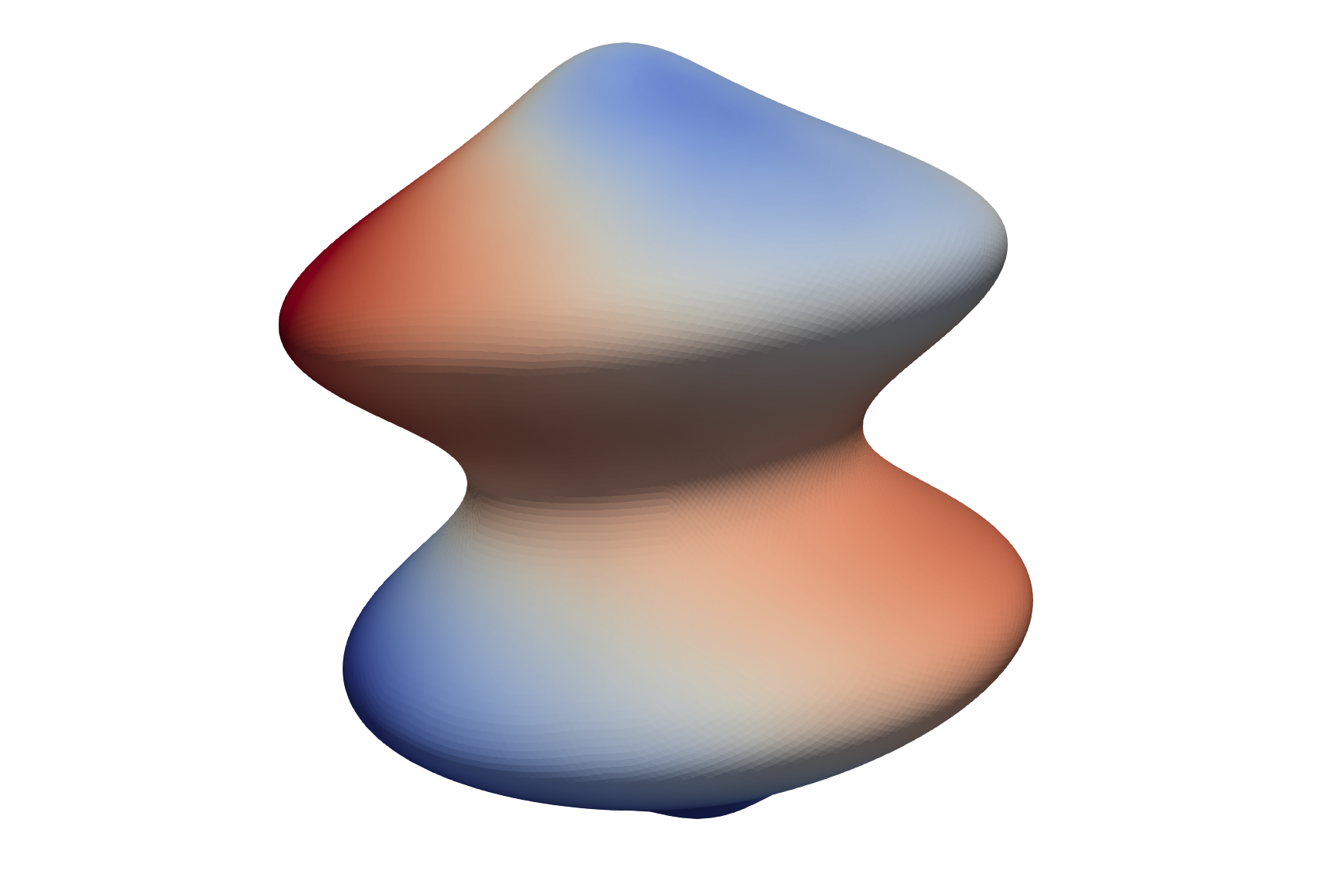}
    }
    \caption{Realizations of \eqref{eq:model-in-intro} for different $\gamma$ and $\Gamma$. The details are given in Example \ref{ex:motivation}. \label{fig:motivation}} 
\end{figure}

The main results of this paper can be summarized as follows:
\begin{enumerate}
    \item we propose a new and computationally efficient surface finite element approximation of \eqref{eq:model-in-intro}, which uses a surface finite element approximation of the additive noise,
    \item we derive strong and pathwise convergence rates for our proposed discretization, which are given in Theorem \ref{theorem:strong-rate} and Corollary \ref{cor:pathwise-rate}, respectively, and finally
    \item we verify the convergence rates obtained by numerical experiments using different values of $\gamma$ in one and two spatial dimensions.
\end{enumerate}
The paper is structured as follows: in Section 2 we state our assumptions on \eqref{eq:model-in-intro} (Assumption \ref{assumption:model}) and its numerical approximation (Assumption \ref{assumption:fem}). We also state a space and time regularity result of the mild solution of \eqref{eq:model-in-intro} (Proposition \ref{proposition:mild-solution-existence}). In Section 3 we outline our surface finite element approximation of \eqref{eq:model-in-intro}. We also state our strong and pathwise convergence results for this approximation (Theorem \ref{theorem:strong-rate} and Corollary \ref{cor:pathwise-rate}). Section 4 contains the proof of Theorem \ref{theorem:strong-rate}, and Section 5 numerical experiments verifying the convergence rate obtained in Theorem \ref{theorem:strong-rate} and Corollary \ref{cor:pathwise-rate}.

\section{Preliminaries and notation}

In what follows, we fix $T > 0$, a filtered probability space $(\Omega, \mathcal{F},\{\mathcal{F}_t\}_{t \in [0,T]}, P)$ (fulfilling the usual conditions), let $H$ be a separable Hilbert space, and $W$ a cylindrical Wiener process on $H$ (covariance operator $I$) adapted to the filtration $\{\mathcal{F}_t\}_{t \in [0,T]}$. Whenever we consider Itô integrals in the following, it will involve this cylindrical Wiener process. For $p \geq 1$ we denote by $L^p(\Omega ; H)$ the Banach space of equivalence classes of measurable functions $(\Omega, \mathcal{F}) \to (H, \mathcal{B}(H))$, with norm
\begin{align*}
    \Vert f \Vert_{L^p(\Omega ; H)}^p := E[\Vert f \Vert_H^p].
\end{align*}
For another separable Hilbert space $U$, we denote by $L(U,H)$ the Banach space of bounded linear operators from $U$ to $H$ with the usual norm, with convention $L(H) := L(H,H)$. We denote by $L_2(U,H)$ the Hilbert space of Hilbert--Schmidt operators from $U$ to $H$, with inner product,
\begin{align*}
    (A,B)_{L_2(U,H)} := \sum_{j = 1}^{\infty}(A e_j, B e_j)_H,
\end{align*}
for any orthonormal basis $\{ e_j \}_{j = 1}^{\infty}$ of $U$, with convention $L_2(H) := L_2(H,H)$.

Throughout the paper, we will denote by $C$ a generic constant, which may change from line to line. Parameter dependence of $C$ will be denoted by subscripts, but is omitted unless relevant. We will also frequently and without further explanation (see, e.g., \ref{cond:A1A2} and \ref{cond:solution-operator}) consider $A \in L(H)$ with the property that $A B \in L(H)$ for some densely defined and possibly unbounded linear $B : D(B) \subseteq H \to H$. By this we understand that $A B$ extends from $D(B)$ to $H$ as a bounded linear operator, and this extension is denoted $AB$.

Finally, we recall some properties of generators of variational semigroups. Let $V$ be another Hilbert space with $V \subseteq H$ densely and continuously, and recall that a sesquilinear form $a : V \times V \to \mathbb{C}$ is continuous with constant $C > 0$ if
\begin{align*}
    \vert a(u, v) \vert \leq C \Vert u \Vert_V \Vert v \Vert_V,
\end{align*}
and coercive with constant $c > 0$ if
\begin{align*}
    \mathrm{Re}(a(u,u)) \geq c \Vert u \Vert_V^2.
\end{align*}
Suppose $a : V \times V \to \mathbb{C}$ is a continuous sesquilinear form, with the property that $\lambda (\cdot, \cdot)_H + a(\cdot, \cdot)$, is coercive for $\lambda \geq 0$ large enough. The operator, $-A : D(A) \to H$, $D(A) \subseteq V$ defined by $(A u, v)_H = a(u, v)$ for any $u \in D(A)$ and $v \in V$ is called the generator of a variational semigroup on $H$, and is a subclass of generators of analytic semigroups. This semigroup is defined by the Dunford integral,
\begin{align*}
    S(t) := \frac{1}{2\pi i} \int_{-\lambda + \gamma} e^{-zt} (z - A)^{-1} \, dz,
\end{align*}
where $\gamma := \{ s e^{\pm i \delta}, \ s \geq 0 \}$, oriented counterclockwise, for some $\delta \in (0, \pi / 2)$ large enough. We can define fractional powers of $\lambda + A$ by, 
\begin{align}\label{eq:fractional-powers}
\begin{split}
    (\lambda + A)^{-\alpha} &:= \frac{1}{\Gamma(\alpha)} \int_0^{\infty} t^{-1 + \alpha} e^{-\lambda t} S(t) \, dt, \\
    (\lambda + A)^{\alpha} &:= \frac{\sin(\alpha \pi)}{\pi} \int_0^{\infty} t^{-1 + \alpha} (t + \lambda + A)^{-1} (\lambda + A) \, dt, 
\end{split}
\end{align}
for $\alpha \in (0,1)$, where $\Gamma(\alpha)$ is the gamma function evaluated at $\alpha$. See, e.g., (6.9) and Theorem 6.9 of Chapter 2 in \cite{pazy} (the expression for $(\lambda + A)^{-\alpha}$ also holds for any $\alpha \geq 1$). The following properties of generators of analytic semigroups will be used frequently. 
\begin{lemma}\label{lemma:analytic-semigroup}
    Let $\lambda$, $A$ be as above, and $S(\cdot)$ the analytic semigroup generated by $-A$. Then,
    \begin{enumerate}
        \item[(a)] for any $\alpha \geq 0$ there are $C_{\alpha}, \delta > 0$ such that for any $t > 0$ and $u \in H$
        \begin{align*}
            S(t) u \in D((\lambda + A)^{\alpha}) \quad \text{with} \quad \Vert (\lambda + A)^{\alpha} S(t) u \Vert_H \leq C_{\alpha} e^{(\lambda-\delta)t} t^{-\alpha} \Vert u \Vert_H,
        \end{align*}
        with convention $(\lambda + A)^0 = I$, \medskip

        \item[(b)] for $\alpha, \beta \in \mathbb{R}$, $u \in D((\lambda + A)^{\gamma})$ with $\gamma = \max(\alpha, \beta, \alpha + \beta)$,
        \begin{align*}
            (\lambda + A)^{\alpha} (\lambda + A)^{\beta} u = (\lambda + A)^{\beta} (\lambda + A)^{\alpha} u = (\lambda + A)^{\alpha + \beta} u,
        \end{align*}
        
        \item[(c)] for $\alpha \in \mathbb{R}$, $D((\lambda + A)^{\alpha})$ is dense in $H$, \medskip
        
        \item[(d)] for $\alpha \in \mathbb{R}$, $(\lambda + A)^{\alpha} S(t) u = S(t) (\lambda + A)^{\alpha} u$ for $u \in D((\lambda + A)^{\alpha})$, \medskip
        
        \item[(e)] and finally for $\alpha \in [0,1]$,
        \begin{align*}
            \Vert (S(t) - I) u \Vert_H &\leq C e^{\lambda t} t^{\alpha} \Vert (\lambda + A)^{\alpha} u \Vert_H.
        \end{align*}
    \end{enumerate}
\end{lemma}
\begin{proof}
    The first four statements can be found in e.g. Chapter 2 (Theorem 6.8 and 6.13) in \cite{pazy}. The last one follows by similar arguments as in the case of $\lambda = 0$ (see, e.g., \cite{2024-auestad}). 
\end{proof}

In what follows, we let $\Gamma \subseteq \mathbb{R}^{d+1}$, $d \leq 3$ be a compact $C^3$-surface, by which we understand:
\begin{enumerate}
    \item $\Gamma \subseteq \mathbb{R}^{d+1}$ is a connected, bounded and boundaryless hypersurface, \medskip
    \item it has unit normal $\nu : \Gamma \to \mathbb{R}^{d+1}$, associated tubular neighborhood $\mathcal{N}_{\epsilon} := \{ x = x_0 + \alpha \nu, \ x_0 \in \Gamma, \ \vert \alpha \vert < \epsilon \}$, and \medskip
    \item a signed distance function $d_s \in C^3(\mathcal{N}_{\epsilon})$, satisfying $\vert d_s(x) \vert := \inf \{ \vert x - x_0 \vert, \ x_0 \in \Gamma \}$.
\end{enumerate}
For a given $\tilde{f} \in C^1(\mathcal{N}_{\epsilon})$ with $f := \tilde{f}\vert_{\Gamma}$ (the restriction of $\tilde{f}$ to $\Gamma$), we define the tangential derivative $\nabla_{\Gamma} f(x) := (I - \nu(x) \nu(x)^T) \nabla \tilde{f}(x)$ for any $x \in \Gamma$. The tangential derivative at $x \in \Gamma$ lies in the tangent plane, denoted $T_x \Gamma$, which is the subspace of $\mathbb{R}^{d+1}$ orthogonal to $\nu(x)$. Denoting by $a_j(x) \in \mathbb{R}^{d+1}$ the $j$'th row of the matrix $I - \nu(x) \nu(x)^T$, the $j$'th component of the tangential derivative may be expressed $\underline{\partial}_j f := a_j(x) \cdot \nabla \tilde{f}$, $x \in \Gamma$, while higher order derivatives may be expressed $\underline{\partial}_i \underline{\partial}_j f := (a_i(x) \cdot \nabla) (a_j(x) \cdot \nabla) \tilde{f}$, provided $\tilde{f}$ and $\nu$ are smooth enough. 

Define $H := L^2(\Gamma)$ and let $V$ be some closed subspace of $H^1(\Gamma)$, where $H^s(\Gamma)$, $s \in \mathbb{N}$, is defined as the completion of $C^s(\mathcal{N}_{\epsilon})$ restricted to $\Gamma$, using the norm given by the inner product, $(u, v)_{H^s} := \sum_{\vert \alpha \vert \leq s}(\underline{\partial}^{\alpha} u, \underline{\partial}^{\alpha} v)_{L^2}$ where $\underline{\partial}^{\alpha} u$ denotes any tangential derivatives of order $\vert \alpha \vert$. Since the components of tangential derivative does not necessarily commute, $\alpha$ is not a multi-index, but we abuse this notation for readability. 

\subsection*{Model and mild solution}

The operators $A_j : D(A_j) \to H$ in \eqref{eq:model-in-intro} are related to the sesquilinear forms $a_j : V \times V \to \mathbb{C}$, where, 
\begin{align*}
    a_j(u, v) := \int_{\Gamma} \mathcal{A}_j \nabla_{\Gamma} u \cdot \nabla_{\Gamma} \overline{v} + (b_j \cdot \nabla_{\Gamma} u) \overline{v} + \alpha_j u \overline{v} \, d \sigma, \quad j = 1, 2,
\end{align*} 
with coefficients $\mathcal{A}_j(x) \in L(T_x \Gamma)$, $b_j(x) \in T_x \Gamma$, $\alpha_j(x) \in \mathbb{R}$, where 
$\overline{v}$ is the complex conjugate of $v$, and $\sigma$ is surface measure on $\Gamma$. We make the following assumptions on \eqref{eq:model-in-intro}.
\begin{assumption}\label{assumption:model}
\hfill
\begin{description}
    \myitem{(M1)}\label{cond:continuity-coercivity} There is $\lambda \geq 0$ such that the shifted sesquilinear form $\lambda (\cdot,\cdot)_H + a_1(\cdot, \cdot)$ and the sesquilinear form $a_2(\cdot, \cdot)$ are coercive and continuous on $V$. \medskip
    \myitem{(M2)}\label{cond:hilbert-schmidt} For any $\alpha < -d/4$, there is $C_{\alpha} > 0$, such that $\Vert A_2^{\alpha} \Vert_{L_2(H)} \leq C_{\alpha}$. \medskip
    \myitem{(M3)}\label{cond:A1A2} For some $C > 0$, $\Vert (\lambda + A_1) A_2^{-1} \Vert_{L(H)} \leq C$ and $\Vert (\lambda + A_1)^{-1/2} A_2^{1/2} \Vert_{L(H)} \leq C$.
    \medskip
    \myitem{(M4)}\label{cond:gamma-size} $\gamma > d/4 - 1/2$. \medskip
    \myitem{(M5)} $\xi$ is $\mathcal{F}_0$-measurable and in $L^p(\Omega ; H)$ for some $p \geq 2$. \medskip
\end{description}
\end{assumption}

\begin{remark}\label{remark:model-assumptions} \hfill
\begin{enumerate}[(a)]
    \item
    Sufficient conditions for \ref{cond:continuity-coercivity} to hold are,
    \begin{enumerate}[(i)]
    \item $\Vert \mathcal{A}_j \Vert_{L(T_x \Gamma)}$, $\Vert b_j \Vert_{\mathbb{R}^{d+1}}$ and $\vert \alpha_j \vert$ are in $L^{\infty}(\Gamma)$, $j = 1, 2$, and 
    \item there is $c > 0$, such that for any $y \in T_x \Gamma$, and $y_0 \in \mathbb{R}$, 
    \begin{align}
    \label{eq:elliptic-condition}
        &y^T \mathcal{A}_1(x) y \geq c \vert y \vert^2, \qquad &&\text{for a.e. } x \in \Gamma, \\
    \label{eq:coercive-condition}
        &y^T \mathcal{A}_2(x) y + b_2(x) \cdot y y_0 + \alpha_2(x) y_0^2 \geq c (\vert y \vert^2 + y_0^2), \qquad &&\text{for a.e. } x \in \Gamma.
    \end{align}
    \end{enumerate}
    In which case one sees that 
    \begin{align*}
        \vert a_j(u, v) \vert \leq C \Vert u \Vert_{H^1} \Vert v \Vert_{H^1}, \quad \mathrm{Re}(a_1(u, u)) + \lambda \Vert u \Vert_{L^2}^2 \geq \frac{c}{2} \Vert u \Vert_{H^1}^2, \quad \mathrm{Re}(a_2(u, u)) \geq c \Vert u \Vert_{H^1}^2,
    \end{align*}
    with $C = 4(\Vert \mathcal{A}_j \Vert_{L^{\infty}} + \Vert b_j \Vert_{L^{\infty}} + \Vert \alpha_j \Vert_{L^{\infty}})$ (with $\Vert \mathcal{A}_j \Vert_{L^{\infty}} := \mathrm{ess \ sup}_{x\in \Gamma}\Vert \mathcal{A}_j(x) \Vert_{L(T_x \Gamma)}$, and similarly for the other terms), provided $\lambda \geq \Vert \alpha_1 \Vert_{L^{\infty}} + \frac{1}{2c} \Vert b_1 \Vert_{L^{\infty}} + \frac{c}{2}$. If $V$ is e.g. the subspace of $H^1(\Gamma)$ which integrates to $0$, \eqref{eq:coercive-condition} can be replaced by the weaker condition $y^T \mathcal{A}_2(x) y + b_2(x) \cdot y y_0 + \alpha_2(x) y_0^2 \geq c \vert y \vert^2$ for a.e. $x \in \Gamma$, since we then have access to a Poincaré inequality (see Theorem 2.12 in \cite{2013-elliott}).

    \item
    By an explicit computation using the spherical harmonics and their eigenvalues, it may be seen that condition \ref{cond:hilbert-schmidt} holds whenever $\Gamma$ is the $d$-dimensional sphere and $A_2 = I - \Delta_{\Gamma}$. More generally, it may be seen that this condition holds for this choice of $A_2$ whenever $\Gamma$ is a compact smooth surface (see Theorem 7.6.4 in \cite{lablee}).
    
    \item
    The first inequality of \ref{cond:A1A2} holds if $D(\lambda + A_1) = D(A_2)$, and $A_2$ satisfies the elliptic regularity estimate $\Vert A_2^{-1} f \Vert_{H^2} \leq C \Vert f \Vert_{L^2}$, $f \in L^2$. This is e.g. the case when $A_1 = -\Delta_{\Gamma}$ and $A_2 = I + A_1$. The second inequality holds if $a_1$ and $a_2$ are symmetric and defined on the same subspaces $V$. Due to \ref{cond:continuity-coercivity} the norms $\Vert (\lambda + A_1)^{1/2} \cdot \Vert_{L^2}, \Vert A_2^{1/2} \cdot \Vert_{L^2}$ on $V$ are equivalent (both are equivalent to the $H^1$-norm), and a standard computation yields the inequality in this case. 
\end{enumerate}
\end{remark}

Under Assumption \ref{assumption:model}, $-A_j : D(A_j) \to H$ are generators of analytic semigroups on $H$, denoted $S_j(\cdot)$, $j = 1, 2$. We are interested in the mild solution of \eqref{eq:model-in-intro}, which is defined by the stochastic convolution
\begin{align}\label{eq:mild-form}
    u(t) = S_1(t) \xi + \int_0^{t} S_1(t - s) A_2^{-\gamma} dW, \quad P\text{-a.s. for any } t \in [0,T].
\end{align}
The next proposition asserts that we have a mild solution under Assumption \ref{assumption:model}, and describes the space and time regularity of this solution. The proof is identical to that of Lemma 2.4, 2.5 and 2.6 in \cite{2025-auestad}.
\begin{proposition}\label{proposition:mild-solution-existence}
    Under Assumption \ref{assumption:model}, \eqref{eq:model-in-intro} has a mild solution, $u$, with $u(t) \in L^p(\Omega ; H)$, for any $t \geq 0$. Further,
    \begin{enumerate}[(i)]
        \item for $\alpha \in [0, \gamma + 1/2 - d/4) \cap [0,3/2)$, and $\rho \leq \alpha$, there are $C_{p,\alpha}, \delta > 0$ such that,
        \begin{align*}
            \Vert (\lambda + A_1)^{\alpha} u(t) \Vert_{L^p(\Omega ; H)} \leq C_{p,\alpha} e^{\lambda t} (1 + e^{-\delta t} t^{-\alpha + \rho} \Vert (\lambda + A_1)^{\rho} \xi \Vert_{L^p(\Omega ; H)}),
        \end{align*}
        \item while for $0 \leq t_1 \leq t_2$, $\alpha \in [0,\gamma + 1/2 - d / 4) \cap [0,1/2]$ and $\rho \leq \alpha$, there are $C_{p,\alpha}, \delta > 0$ such that, 
        \begin{align*}
            \Vert u(t_2) - u(t_1) \Vert_{L^p(\Omega ; H)} \leq C_{p,\alpha} e^{\lambda t_2} (t_2-t_1)^{\alpha} (1 + e^{-\delta t_1} t_1^{-\alpha + \rho} \Vert (\lambda +  A_1)^{\rho} \xi \Vert_{L^p(\Omega ; H)}).
        \end{align*}
        \end{enumerate}
\end{proposition}
As a consequence of Proposition \ref{proposition:mild-solution-existence} $(i)$ above, $u(t) \in D((\lambda + A_1)^{\alpha}), \, t > 0$, $P$-a.s. for any $\alpha \in [0, \gamma + 1/2 - d/4) \cap [0,3/2)$. By Proposition \ref{proposition:mild-solution-existence} $(ii)$ one can deduce that $u$ has $\beta$-Hölder continuous trajectories on any interval $[t_1, t_2]$, $0 < t_1 < t_2 < \infty$ $P$-a.s. for any $\beta \in [0,\gamma + 1/2 - d/4) \cap [0,1/2)$. To see this, note that the Hölder regularity follows when $\xi = 0$ by the Kolmogorov continuity test (see, e.g., Theorem 3.3 in \cite{da-prato}). Since $S(\cdot) \xi$ is Lipchitz on the interval $[t_1, t_2]$ $P$-a.s. (by Lemma \ref{lemma:analytic-semigroup}), the $\beta$-Hölder regularity of the sum follows. 

\subsection*{Surface finite element approximation}

In order to approximate $A_j$, we use the surface finite element approximation introduced in \cite{1988-dziuk}. To that end, let $\Gamma_h \subset \mathcal{N}_{\epsilon}$ be a discrete surface, consisting of the union of simplices $\tau \in \mathcal{T}_h$ with disjoint interior, maximum diameter $h$, and vertices on $\Gamma$. The simplices $\mathcal{T}_h$ are regular in the sense that for $C > 0$,
\begin{align}\label{eq:quasi-uniform}
    C^{-1} h \leq 2 \rho(\tau) \leq \text{diam}(\tau) \leq 2 r(\tau) \leq C h, \quad \tau \in \mathcal{T}_h,
\end{align}
where $\rho, r$ are the radii of the incircle and circumcircle, respectively. For any point $x \in \Gamma_h$, we can associate a unique point $p(x) \in \Gamma$ by the mapping $p(x) := x - d(x) \nabla d(x)$. For any $f : \Gamma_h \to \mathbb{R}$, we define the lift $\iota_h f := f \circ p\vert_{\Gamma_h}^{-1} : \Gamma \to \mathbb{R}$. Moreover, we define the quotient of surface area measure, $\delta_h \in L^1(\Gamma_h)$, by $\int_U \delta_h \, d \sigma_h = \int_{p(U)} \, d\sigma$ for any $U \subseteq \Gamma_h$ measurable, where $\sigma_h$ is surface measure on $\Gamma_h$. 

We let $V_h \subseteq C(\Gamma_h)$ be the $N_h$-dimensional linear space consisting of continuous functions on $\Gamma_h$ that are first order polynomials when restricted to any $\tau \in \mathcal{T}_h$. For $u, v \in V_h$, we define the discrete sesquilinear forms,
\begin{align*}
    a_{j,h}(u, v) := \int_{\Gamma_h} \mathcal{A}_{j,h} \nabla_{\Gamma_h} u \cdot \nabla_{\Gamma_h} \overline{v} + (b_{j,h} \cdot \nabla_{\Gamma_h} u) \overline{v} + \alpha_{j,h} u \overline{v} \, d \sigma_h, \quad j = 1, 2,
\end{align*}
with coefficients $\mathcal{A}_{j,h}(x) \in L(T_x \Gamma_h)$, $b_{j,h}(x) \in T_x \Gamma_h$, $\alpha_{j,h}(x) \in \mathbb{R}$ defined on $\Gamma_h$, where the tangential derivatives on $\Gamma_h$ are defined similarly as those on $\Gamma$ (but only in an almost everywhere sense). The precise conditions on the sesquilinear forms and their coefficients are listed in Assumption \ref{assumption:fem}.

It will be convenient to define $H^1(\Gamma_h)$ as the completion of $C^1(\mathcal{N}_{\epsilon})$ restricted to $\Gamma_h$ using the norm given by the inner product
\begin{align*}
    (u, v)_{H^1(\Gamma_h)} := \int_{\Gamma_h} u \overline{v} + \nabla_{\Gamma_h} u \cdot \nabla_{\Gamma_h} \overline{v} \, d\sigma_h.
\end{align*}
The following lemma will be useful.
\begin{lemma}\label{lemma:sfem}
\begin{enumerate}[(a)] 
    With $\Gamma_h$ as above, \medskip
    \item there is $C > 0$ such that for any $u \in V_h$, $\Vert u \Vert_{H^1(\Gamma_h)} \leq C h^{-1} \Vert u \Vert_{L^2(\Gamma_h)}$, \medskip
    \item for $C > 0$, $\Vert 1-\delta_h \Vert_{L^{\infty}(\Gamma_h)} \leq C h^2$, and finally, \medskip
    \item the norms,
    \begin{align*}
        \Vert \cdot \Vert_{H^1(\Gamma)}, \quad \Vert \iota_h^{-1} \cdot \Vert_{H^1(\Gamma_h)},
    \end{align*}
    are both equivalent on $H^1(\Gamma)$, with constants independent of $h$. 
\end{enumerate}
\end{lemma}
\begin{proof}
    (a) follows by \eqref{eq:quasi-uniform} and Theorem 3.2.6 in \cite{ciarlet} (noting that it does not matter that our simplices are embedded in $\mathbb{R}^{d+1}$), for (b) and (c) see Lemma 4.1 and 4.2 in \cite{2013-elliott}. 
\end{proof}
We define the operator $\tau_h \in L(H)$ by 
\begin{align}\label{eq:scaling-by-quotient}
    \tau_h f := (\iota_h \delta_h^{1/2}) f,
\end{align}
i.e. scaling by the (lifted) square root of the quotient of surface measure. A consequence of Lemma \ref{lemma:sfem} (b)\footnote{Noting that $\vert 1 - \delta_h^{1/2} \vert = \vert (1 + \delta_h^{1/2})^{-1}(1 - \delta_h) \vert \leq \vert 1 - \delta_h \vert$.} above is the bound
\begin{align}\label{eq:quotient-scaling-estimate}
    \Vert I - \tau_h \Vert_{L(H)} \leq C \Vert 1 - \delta_h^{1/2} \Vert_{L^{\infty}(\Gamma_h)} \leq C h^2.
\end{align}

To $a_{j,h}$ we associate discrete operators $A_{j,h}' : V_h \to V_h$ satisfying $(A_{j,h}' u, v)_{L^2(\Gamma_h)} = a_{j,h}(u, v)$, for any $u, v \in V_h$. Finally, we denote the $L^2(\Gamma_h)$-orthogonal projection onto $V_h$ by $\mathcal{P}_h$, and define the finite dimensional operators,
\begin{align}\label{eq:Ajh}
    A_{j,h} := \iota_h A_{j,h}' \iota_h^{-1} : \iota_h V_h \to \iota_h V_h,
\end{align}
\begin{align}\label{eq:pih}
    \pi_h := \iota_h \mathcal{P}_h \iota_h^{-1} : L^2(\Gamma) \to \iota_h V_h.
\end{align}
We make the following assumption on the surface finite element approximation $A_{j,h}$ of $A_j$. 
\begin{assumption}\label{assumption:fem}
\hfill
\begin{description}
    \myitem{(N1)}\label{cond:discrete-operator} With $\lambda \geq 0$ as in \ref{cond:continuity-coercivity} the shifted sesquilinear form $\lambda (\cdot, \cdot)_{L^2(\Gamma_h)} + a_{1,h}(\cdot, \cdot)$, and the sesquilinear form $a_{2,h}(\cdot, \cdot)$ are coercive and continuous on $(V_h, \Vert \cdot \Vert_{H^1(\Gamma_h)})$, with coercivity and continuity constant independent of $h$. \medskip
    \myitem{(N2)}\label{cond:solution-operator} With $\lambda \geq 0$ as in \ref{cond:continuity-coercivity},
    \begin{align*}
        \Vert ((\lambda + A_1)^{-1} - (\lambda + A_{1,h})^{-1} \pi_h) (\lambda + A_1)^{\alpha} \Vert_{L(H)} \leq C h^{2-2\alpha}, 
    \end{align*}
    and
    \begin{align*}
        \Vert A_2^{\alpha} (A_2^{-1} - A_{2,h}^{-1} \pi_h) \Vert_{L(H)} \leq C h^{2 - 2\alpha},
    \end{align*}
    for $\alpha \in \{0, 1/2\}$, and some $C > 0$. \medskip
    \myitem{(N3)}\label{cond:fem-bound} There is $C > 0$ such that $\Vert (\lambda + A_{1,h})^{-1/2} \pi_h (\lambda + A_1)^{1/2} \Vert_{L(H)} \leq C$.
\end{description}
\end{assumption}

\begin{remark}\label{remark:numerics-assumptions} \hfill
\begin{enumerate}[(a)]
    \item 
    Condition \ref{cond:discrete-operator} holds under analogous conditions on the coefficients as those in Remark \ref{remark:model-assumptions} (a). This condition implies that the operators $A_{j,h} : \iota_h V_h \to \iota_h V_h$ are uniformly sectorial in the sense of condition (A4) in \cite{2024-auestad}. To see this, it suffices to verify that the sesquilinear forms $((\lambda + A_{1,h}) \cdot, \cdot)_H$ and $(A_{2,h} \cdot, \cdot)_H$ are coercive and continuous on $(\iota_h V_h, \Vert \cdot \Vert_{H^1(\Gamma)})$ with continuity and coercivity constants independent of $h$. To that end, let $u, v \in \iota_h V_h$, and note that
    \begin{align*}
        \lambda (u, v)_H = \lambda (\iota_h^{-1} u, (\iota_h^{-1} v) (\delta_h - 1))_{L^2(\Gamma_h)} + \lambda (\iota_h^{-1} u, \iota_h^{-1} v)_{L^2(\Gamma_h)},
    \end{align*}
    with $\vert \lambda (\iota_h^{-1} u, (\iota_h^{-1} v) (\delta_h - 1))_{L^2(\Gamma_h)} \vert \leq C h^2 \Vert u \Vert_H \Vert v \Vert_H$ by Lemma \ref{lemma:sfem} (b), and
    \begin{align*}
        (A_{j,h} u, v)_H &= (A_{j,h}' \iota_h^{-1} u, (\iota_h^{-1} v) \delta_h )_{L^2(\Gamma_h)} \\
        &= (A_{j,h}' \iota_h^{-1} u, \mathcal{P}_h (\iota_h^{-1} v) \delta_h )_{L^2(\Gamma_h)} \\
        &= a_{j,h}(\iota_h^{-1} u, \mathcal{P}_h (\iota_h^{-1} v) \delta_h) \\
        &= a_{j,h}(\iota_h^{-1} u, \mathcal{P}_h (\iota_h^{-1} v) (\delta_h - 1)) + a_{j,h}(\iota_h^{-1} u, \iota_h^{-1} v)
    \end{align*}
    where we used the definition of $A_{j,h}'$. For the first term above, note that by \ref{cond:discrete-operator} and Lemma \ref{lemma:sfem}
    \begin{align*}
        \vert a_{j,h}(\iota_h^{-1} u, \mathcal{P}_h (\iota_h^{-1} v) (\delta_h - 1)) \vert &\leq C \Vert u \Vert_{H^1(\Gamma)} \Vert \mathcal{P}_h (\iota_h^{-1} v) (\delta_h - 1) \Vert_{H^1(\Gamma_h)} \\
        &\leq C \Vert u \Vert_{H^1(\Gamma)} h^{-1} \Vert \mathcal{P}_h (\iota_h^{-1} v) (\delta_h - 1) \Vert_{L^2(\Gamma_h)} \\
        &\leq C \Vert u \Vert_{H^1(\Gamma)} h^{-1} \Vert (\iota_h^{-1} v) (\delta_h - 1) \Vert_{L^2(\Gamma_h)} \\
        &\leq C \Vert u \Vert_{H^1(\Gamma)} h^{-1} \Vert (\iota_h^{-1} v) \Vert_{L^2(\Gamma_h)} \Vert (\delta_h - 1) \Vert_{L^{\infty}(\Gamma_h)} \\
        &\leq C h \Vert u \Vert_{H^1(\Gamma)} \Vert v \Vert_{H^1(\Gamma)}.
    \end{align*}
    The second term combined with the term $\lambda (\iota_h^{-1} \cdot, \iota_h^{-1} \cdot)_{L^2(\Gamma_h)}$ is coercive and continuous on $(\iota_h V_h, \Vert \cdot \Vert_{H^1(\Gamma)})$ with constants independent of $h$ by \ref{cond:discrete-operator} and Lemma \ref{lemma:sfem} (c). Thus, for $h$ small enough, $((\lambda + A_{1,h}) \cdot, \cdot)$ and $(A_{2,h} \cdot, \cdot)$ are coercive and continuous on $(\iota_h V_h, \Vert \cdot \Vert_{H^1(\Gamma)})$ with continuity and coercivity constants independent of $h$. 
    \item
    We interpret \ref{cond:solution-operator} as a condition on the rate of convergence of the error $u - u_h$, where $(\lambda + A_1) u = f$ while $(\lambda + A_{1,h}) u_h = \pi_h f$. In the first inequality of \ref{cond:solution-operator}, different $\alpha$ corresponds to different regularity of $f$, while in the second it corresponds to different norms used for the difference $u - u_h$. 
    \item
    Condition \ref{cond:discrete-operator}--\ref{cond:fem-bound} holds in the case of $A_1 = -\Delta_{\Gamma}$, $A_2 = I - \Delta_{\Gamma}$, (with $\mathcal{A}_{j,h} = I, b_{j,h} = 0$, $j = 1, 2$ and $\alpha_{2,h} = 1$), as shown in Section 3 in \cite{2024-auestad}.
\end{enumerate}
\end{remark}

\subsection*{Quadrature approximation of fractional operator}

In order to approximate the fractional power operator, we use the quadrature in \cite{2019-bonito}. To that end, let
\begin{align}\label{eq:quadrature}
    Q_k^{-\gamma}(A_2) := \frac{k \sin(\pi \gamma)}{\pi} \sum_{j = -M}^N e^{(1-\gamma) j k } (e^{j k} I + A_2)^{-1}.
\end{align}
Here, $k > 0$ is the quadrature resolution, and
\begin{align*}
    N = \bigg\lceil \frac{\pi^2}{2 \gamma k^2} \bigg\rceil, \quad M = \bigg\lceil \frac{\pi^2 }{2(1 - \gamma) k^2} \bigg\rceil.
\end{align*}
The following lemma describes the convergence rate of this approximation. 
\begin{lemma}\label{lemma:A_h-quadrature}
Let $A_{2,h}$ be as above, $Q_k^{-\gamma}$ be given by \eqref{eq:quadrature}, and $\gamma \in (0,1)$. Then, for some $C > 0$ independent of $k$ and $h$, the following estimate holds,
\begin{align}
    \Vert (A_{2,h}^{-\gamma} - Q_k^{-\gamma}(A_{2,h})) u \Vert_H \leq C e^{-\frac{\pi^2}{2k}} \Vert u \Vert_H, \quad \text{for any } u \in \iota_h V_h.
\end{align}
\end{lemma}
\begin{proof}
    As stated in Theorem 3.2 in \cite{2019-bonito}, the constant $C$ depends only on $\gamma$ and the continuity and coercivity constant of $(A_{2,h} \cdot, \cdot)_H$ on $(\iota_h V_h, \Vert \cdot \Vert_{H^1(\Gamma)})$, which are independent of $h$ by \ref{cond:discrete-operator} (see Remark \ref{remark:numerics-assumptions} (a)).
\end{proof}

\section{Numerical method and convergence results}\label{section:numerical-method}

In this section we present our proposed finite element approximation, and state our convergence results. Our semidiscrete approximation is based on the following SPDE with values in $\iota_h V_h$,
\begin{align}\label{eq:mild-solution-semi-approx}
    d u_h(t) = -A_{1,h} u_h(t) \, dt + A_{2,h}^{-\gamma} \pi_h \tau_h \, dW(t), \quad u_h(0) = \pi_h \xi.
\end{align}
When $\gamma \in (0,1)$ we approximate the fractional power operator by the quadrature approximation \eqref{eq:quadrature}, in which case \eqref{eq:mild-solution-semi-approx} becomes
\begin{align}\label{eq:semidiscrete-gamma-in-0-1}
    d u_h(t) = -A_{1,h} u_h(t) \, dt + Q_k^{-\gamma}(A_{2,h}) \pi_h \tau_h \, dW(t),
\end{align}
with convention $Q_k^{-1}(A_{2,h}) = A_{2,h}^{-1}$ and $Q_k^0(A_{2,h}) = I$. Discretizing \eqref{eq:semidiscrete-gamma-in-0-1} in time with backward Euler, we get our fully discrete approximation,
\begin{align}\label{eq:fully-discrete-scheme}
    (I + \Delta t A_{1,h})u_{h,\Delta t}(t_{n+1}) = u_{h,\Delta t}(t_n) + Q_k^{-\gamma}(A_{2,h}) \pi_h \tau_h (W(t_{n+1}) - W(t_n)),
\end{align}
with $t_n = n \Delta t$, $\Delta t = T / N$ for some $N > 0$, and we define $u_{h,\Delta t}(t), t \in (t_n, t_{n+1})$ by \eqref{eq:fully-discrete-2}. 

In order to simulate from \eqref{eq:fully-discrete-scheme} we rewrite this equation as a system of equations for the coefficients of $\iota_h^{-1} u_{h,\Delta t}$ in the nodal basis of $V_h$. Denote this basis by $\varphi_j, \ j = 1, \dots, N_h$, and let
\begin{align}\label{eq:fem-matrices}
    (M_h)_{ij} = \int_{\Gamma_h} \varphi_j \varphi_i \, d\sigma_h, \quad (T_{h})_{ij} = a_{1,h}(\varphi_j, \varphi_i), \quad (K_h)_{ij} = a_{2,h}(\varphi_j, \varphi_i).
\end{align}
We may express $u_{h,\Delta t}(t_n) = \iota_h \sum_{j = 1}^{N_h} \alpha_j^n \varphi_j$, for coefficients $\alpha_j^n$, and using Lemma \ref{lemma:discrete-cylindrical-wiener-process} (see Appendix \ref{app:derivation-of-system} for a detailed derivation), we note that \eqref{eq:fully-discrete-scheme} may be rewritten as in \eqref{eq:scheme-basis-coefficients}.

The following theorem is the main result of this paper, and describes rates of strong convergence of our approximation \eqref{eq:fully-discrete-scheme}.
\begin{theorem}\label{theorem:strong-rate}
Suppose Assumption \ref{assumption:model} and \ref{assumption:fem} hold, $\gamma \in (d/4 - 1/2, 1] \cap [0,1]$, the quadrature resolution $k \leq -\frac{\pi^2}{2}(2 \gamma + 1)^{-1} \log(h)^{-1}$, and let $u$ and $u_{h,\Delta t}$ be the solutions to \eqref{eq:mild-form} and \eqref{eq:fully-discrete-scheme}, respectively. Then for any $\theta \in [0, 2 \gamma + 1 - d/2) \cap [0,2]$ and $\rho \in [-1, \theta] \cap [-2 + \theta, \theta]$, there are $C_{p,\theta}, c > 0$ such that
\begin{align*}
    \Vert u(t) - u_{h, \Delta t}(t) \Vert_{L^p(\Omega; H)} \leq C_{p,\theta} e^{c \lambda t} (h^{\theta} + \Delta t^{\theta / 2}) (1 + t^{-\theta / 2 + \rho / 2} \Vert (\lambda + A_1)^{\rho / 2} \xi \Vert_{L^p(\Omega ; H)}).
\end{align*}
\end{theorem}
A consequence of the $L^p$ convergence of Theorem \ref{theorem:strong-rate} is essentially the same rate of pathwise convergence, as described in the following corollary. 
\begin{corollary}\label{cor:pathwise-rate}
    Suppose the conditions of Theorem \ref{theorem:strong-rate} hold. Then, for any $\theta \in [0, 2\gamma + 1 - d/2) \cap [0,2]$, $\rho \in [-1, \theta] \cap [-2 + \theta, \theta]$, $\epsilon > 0$ small, and sequences $h_n, \Delta t_n$ such that $\sum_{n = 1}^{\infty} (h_n^{\theta} + \Delta t_n^{\theta/2})^p < \infty$ for some $p > 0$, there is a random variable $M_{\theta, \epsilon} > 0$ ensuring
    \begin{align*}
        \Vert u(t) - u_{h_n, \Delta t_n}(t) \Vert_H \leq e^{c\lambda t}(C t^{-\theta / 2 + \rho / 2} \Vert (\lambda + A_1)^{\rho / 2} \xi \Vert_H + M_{\theta, \epsilon}) (h_n^{\theta} + \Delta t_n^{\theta / 2})^{1-\epsilon}, \quad P\text{-a.s.}
    \end{align*}
\end{corollary}
\begin{proof}
    This follows by noting that
    \begin{align*}
        u(t) - u_{h_n, \Delta t_n}(t) &= (S_1(t) - S_{h,\Delta t} \pi_h) \xi + \int_0^t (S_1(t-s) A_2^{-\gamma} - S_{h,\Delta t} Q_k^{-\gamma}(A_{2,h}) \pi_h \tau_h) \, dW \\
        &= (i) + (ii),
    \end{align*}
    and applying Lemma \ref{lemma:fully-discrete-semigroup-approximation} to $(i)$, and Theorem \ref{theorem:strong-rate} combined with Lemma 4.5 in \cite{2025-auestad} to $(ii)$. 
\end{proof}

\begin{remark}
    It is worth noting that the formulation of the surface finite element approximation \eqref{eq:fully-discrete-scheme} differs from that commonly used in literature (see, e.g., \cite{2013-elliott}). In the surface finite element literature, one commonly refers to $\iota_h^{-1} u_{h,\Delta t}$ as the surface finite element approximation, while $u_{h,\Delta t}$ is the ``lifted" surface finite element approximation. Similarly, and in contrast to \eqref{eq:mild-solution-semi-approx}, the equation defining the semidiscrete approximation would involve $\iota_h^{-1} u_h$ and slightly different operators $A_{j,h}$ and $\pi_h$. Our reason for choosing the formulation \eqref{eq:mild-solution-semi-approx} is that it simplifies the convergence analysis: for one, estimates related to the resolvent of $A_{j,h}$ in Assumption \ref{assumption:fem} are directly transferrable to the abstract analysis in \cite{2024-auestad}, making the estimates of Lemma \ref{lemma:fully-discrete-semigroup-approximation}, \ref{lemma:fully-discrete-semigroup-smoothing} and \ref{lemma:new-norm-semidiscrete-error} available, which are essential components in the convergence analysis of our discretization (see the proof of Theorem \ref{theorem:strong-rate}). Moreover, it avoids excessive notation and use of ``lifts" in the proof of Theorem \ref{theorem:strong-rate}, effectively making it almost identical to the convergence analysis in the flat domain case (see the proof of Theorem 3.1 in \cite{2025-auestad}), with the exception of the additional operator $\tau_h$. Both the traditional and proposed formulations give the same fully discrete scheme involving the nodal basis coefficients \eqref{eq:scheme-basis-coefficients}, which is what is implemented and used to simulate.
\end{remark}

\section{Proof of Theorem \ref{theorem:strong-rate}}

To prove Theorem \ref{theorem:strong-rate}, we need a couple of lemmas. Let $S_{h,\Delta t}(\cdot)$ be our fully discrete approximation of $S_1(\cdot)$ based on backward Euler, defined by
\begin{align}\label{eq:fully-discrete-semigroup-operator}
    S_{h,\Delta t}(t) :=
    \begin{cases}
        I, \quad &t = 0, \\
        r(\Delta t A_{1,h} )^{n+1}, \quad &t \in (t_n,t_{n+1}],
    \end{cases}
\end{align}
where $r(z) := (1 + z)^{-1}$, $\chi_D(\cdot)$ is the indicator function on $D$. The following lemma describes the error of this approximation. 
\begin{lemma}\label{lemma:fully-discrete-semigroup-approximation}
    Let $S_{h,\Delta t}(t)$ be given as in \eqref{eq:fully-discrete-semigroup-operator}. Under Assumption \ref{assumption:model} and \ref{assumption:fem}, there are $C, c, \delta > 0$ such that
    \begin{align*}
        \Vert (S_1(t) - S_{h,\Delta t}(t) \pi_h ) u \Vert_H &\leq C e^{c (\lambda-\delta) t} t^{-\theta/2 + \rho/2} (h^{\theta} + \Delta t^{\theta / 2}) \\
        &\qquad \times \min(\Vert (\lambda + A_1)^{\rho/2} u \Vert_H, \Vert A_2^{\rho / 2} u \Vert_H),
    \end{align*}
    for $\theta \in [0,2]$, and $\rho \in [-1, \theta] \cap [-2 + \theta, \theta]$.
\end{lemma}
\begin{proof}
    This follows by Theorem 2.14 and 2.24 in \cite{2024-auestad}, noting that Assumption \ref{assumption:model} and \ref{assumption:fem} imply Assumption 2.1 and condition (A6)--(A9) in \cite{2024-auestad}. 
    
    We may replace $\lambda + A_1$ by $A_2$ since we can recover the inequality of the lemma by interpolating inequalities related to the operator $(S_1(t) - S_{h,\Delta t}(t) \pi_h ) (\lambda + A_1)^{\alpha}$ with $\alpha \in \{-1, 0, 1/2\}$, (in which case $(\lambda + A_1)^{\alpha}$ may be replaced by the corresponding powers of $A_2$ due to condition \ref{cond:A1A2}). 
\end{proof}
The next lemma describes a smoothing property for the fully discrete semigroup, $S_{h,\Delta t}(\cdot)$, and follows from Lemma \ref{lemma:analytic-semigroup}, \ref{lemma:fully-discrete-semigroup-approximation}, \ref{cond:fem-bound} and \ref{cond:A1A2}.
\begin{lemma}\label{lemma:fully-discrete-semigroup-smoothing}
    Under the conditions of Lemma \ref{lemma:fully-discrete-semigroup-approximation}, there are $c, C, \delta > 0$ such that
    \begin{align*}
        \Vert S_{h,\Delta t}(t) \pi_h A_2^{\alpha} \Vert_{L(H)} \leq C e^{c(\lambda - \delta) t} t^{-\alpha},
    \end{align*}
    for $\alpha \in [0, 1/2]$.
\end{lemma}
\begin{proof}
    By the proofs of Lemma 2.7 and 2.10 in \cite{2024-auestad}, condition \ref{cond:fem-bound} and \ref{cond:A1A2},
    \begin{align*}
        \Vert S_{h,\Delta t}(t) \pi_h A_2^{\alpha} \Vert_{L(H)} &\leq \Vert S_{h,\Delta t}(t) (\lambda + A_{1,h})^{\alpha} \pi_h \Vert_{L(H)} \Vert (\lambda + A_{1,h})^{-\alpha} \pi_h (\lambda + A_1)^{\alpha} \Vert_{L(H)} \\
        &\qquad \times \Vert (\lambda + A_1)^{-\alpha} A_2^{\alpha} \Vert_{L(H)} \\
        &\leq C e^{c (\lambda-\delta) t} t^{-\alpha},
    \end{align*}
    for $\alpha \in \{0,1/2\}$. Therefore, by interpolation, it holds for $\alpha \in [0,1/2]$.
\end{proof}
The following lemma is an error estimate for the semidiscrete approximation, $S_{2,h}(\cdot)$, of $S_2(\cdot)$. 
\begin{lemma}\label{lemma:new-norm-semidiscrete-error}
    Let $S_{2,h}(\cdot)$ be the analytic semigroup generated by $-A_{2,h}$ on $\iota_h V_h$. Under Assumption \ref{assumption:model} and \ref{assumption:fem} there are $C, \delta > 0$ such that 
    \begin{align*}
        \Vert A_2^{\alpha} (S_2(t) - S_{2,h}(t) \pi_h) \Vert_{L(H)} \leq C e^{-\delta t} t^{-\theta / 2 } h^{\theta - 2 \alpha},
    \end{align*}
    for any $\alpha \in [-1/2, 1/2]$, $\theta \in [2\alpha,2] \cap [0,2+2\alpha]$. 
\end{lemma}
\begin{proof}
    By similar arguments as in the proof of Lemma \ref{lemma:fully-discrete-semigroup-approximation}, this follows by Lemma 2.15 and 2.18 in \cite{2024-auestad}.
\end{proof}
The following lemma describes the error in our approximation of the stochastic convolution \eqref{eq:mild-form} due to the surface finite element approximation of the additive noise. 
\begin{lemma}\label{lemma:error-stochastic-convolution}
    Suppose the conditions of Theorem \ref{theorem:strong-rate} holds. Then for any $p \geq 2$ and $\theta \in [0, 2\gamma + 1 - d / 2) \cap [0,2]$, there are $C_{p,\theta}, c > 0$ such that
    \begin{align*}
        \Vert \int_0^t S_1(t-s) A_2^{-\gamma} \, dW - \int_0^t S_{h,\Delta t}(t-s) A_{2,h}^{-\gamma} \pi_h \tau_h \, dW \Vert_{L^p(\Omega ; H)} \leq C_{p,\theta} e^{c \lambda t} (h^{\theta} + \Delta t^{\theta / 2}).
    \end{align*}
\end{lemma}
\begin{proof}
    We decompose the difference as follows
    \begin{align*}
        &\int_0^t S_1(t-s) A_2^{-\gamma} \, dW - \int_0^t S_{h,\Delta t}(t-s) A_{2,h}^{-\gamma} \pi_h \tau_h \, dW \\
        &\qquad = \int_0^t (S_1(t-s) - S_{h,\Delta t}(t-s) \pi_h) A_2^{-\gamma} \, dW \quad (=: (i)) \\
        &\qquad \qquad + \int_0^t S_{h,\Delta t}(t-s) \pi_h (A_2^{-\gamma} - A_{2,h}^{-\gamma} \pi_h \tau_h) \, dW. \quad (=: (ii))
    \end{align*}
    From Lemma 4.4 in \cite{2025-auestad} (see the treatment of the term $(i)$) we have for some $0 < \epsilon < (\gamma - d/4 + 1/2) / 2$
    \begin{align*}
        \Vert (i) \Vert_{L^p(\Omega ; H)} \leq C_{p,\epsilon} e^{c \lambda t} (h^{2 \gamma + 1 - d/2 - 4\epsilon} + \Delta t^{\gamma + 1 / 2 - d / 4 - 2\epsilon}).
    \end{align*}
    
    For the second term, Lemma \ref{lemma:fully-discrete-semigroup-smoothing}, the Burkholder--Davis--Gundy inequality (Theorem 4.36 in \cite{da-prato}), and the properties $\Vert L \Vert_{L_2(H)} = \Vert L^* \Vert_{L_2(H)}$, $\Vert L \Vert_{L(H)} = \Vert L^* \Vert_{L(H)}$ for any $L \in L_2(H)$ with adjoint $L^*$, gives for $0 < \epsilon < (\gamma - d / 4 + 1/2) / 2$
    \begin{align*}
        \Vert (ii) \Vert_{L^p(\Omega ; H)}^2 &\leq C_p \int_0^t \Vert S_{h,\Delta t}(t-s) \pi_h (A_2^{-\gamma} - A_{2,h}^{-\gamma} \pi_h \tau_h) \Vert_{L_2(H)}^2 \, ds \\
        &= C_p \int_0^t \Vert S_{h,\Delta t}(t-s) \pi_h A_2^{1/2 - \epsilon} A_2^{-1/2 + \epsilon} (A_2^{-\gamma} - A_{2,h}^{-\gamma} \pi_h \tau_h) \Vert_{L_2(H)}^2 \, ds \\
        &\leq C_p \int_0^t \Vert S_{h,\Delta t}(t-s) \pi_h A_2^{1/2 - \epsilon} \Vert_{L(H)}^2 \Vert A_2^{-1/2 + \epsilon} (A_2^{-\gamma} - A_{2,h}^{-\gamma} \pi_h \tau_h) \Vert_{L_2(H)}^2 \, ds \\
        &\leq C_{p,\epsilon} e^{2 c \lambda t} \Vert A_2^{-1/2 + \epsilon} (A_2^{-\gamma} - A_{2,h}^{-\gamma} \pi_h \tau_h) \Vert_{L_2(H)}^2 \\
        &= C_{p,\epsilon} e^{2 c \lambda t} \Vert (A_2^{d/4 -1/2 + 2\epsilon} (A_2^{-\gamma} - A_{2,h}^{-\gamma} \pi_h \tau_h) )^* (A_2^{-d/4-\epsilon})^* \Vert_{L_2(H)}^2 \\
        &\leq C_{p,\epsilon} e^{2 c \lambda t} \Vert A_2^{d/4 -1/2 + 2\epsilon} (A_2^{-\gamma} - A_{2,h}^{-\gamma} \pi_h \tau_h) \Vert_{L(H)}^2 \Vert A_2^{-d/4 - \epsilon} \Vert_{L_2(H)}^2 \\
        &\leq C_{p,\epsilon} e^{2 c \lambda t} \Vert A_2^{d/4 -1/2 + 2\epsilon} (A_2^{-\gamma} - A_{2,h}^{-\gamma} \pi_h \tau_h) \Vert_{L(H)}^2,
    \end{align*}
    due to condition \ref{cond:hilbert-schmidt}. Note that,
    \begin{align*}
        &\Vert A_2^{d/4 -1/2 + 2\epsilon} (A_2^{-\gamma} - A_{2,h}^{-\gamma} \pi_h \tau_h) \Vert_{L(H)}^2 \\
        &\quad \leq 2 \Vert A_2^{d/4 -1/2 + 2\epsilon} (A_2^{-\gamma} - A_{2,h}^{-\gamma} \pi_h) \Vert_{L(H)}^2 + 2 \Vert A_2^{d/4 -1/2 + 2\epsilon} A_{2,h}^{-\gamma} \pi_h (I - \tau_h) \Vert_{L(H)}^2.
    \end{align*}
    For the first term above, note that by the definition of the negative fractional power \eqref{eq:fractional-powers} and Lemma \ref{lemma:new-norm-semidiscrete-error},
    \begin{align*}
        \Vert A_2^{d/4 -1/2 + 2\epsilon} (A_2^{-\gamma} - A_{2,h}^{-\gamma} \pi_h) \Vert_{L(H)} &= \Vert A_2^{d/4 - 1/2 + 2 \epsilon} \int_0^{\infty} t^{-1 + \gamma} (S_2(t) - S_{2,h}(t) \pi_h) \, dt \Vert_{L(H)} \\
        &\leq \int_0^{\infty} t^{-1 + \gamma} \Vert A_2^{d/4 - 1/2 + 2 \epsilon} (S_2(t) - S_{2,h}(t) \pi_h) \Vert_{L(H)} \, dt \\
        &\leq C \int_0^{\infty} e^{-\delta t} t^{-1 + \gamma - \theta / 2} h^{\theta - 2(d/4-1/2+2\epsilon)} \, dt \\
        &\leq C_{\epsilon} h^{2 \gamma + 1 - d / 2 - 5 \epsilon},
    \end{align*}
    with $\theta = 2 \gamma - \epsilon$ in the last line, where we used that $A_2^{\alpha}$ is closed to pass it under the integral sign. For the second term we get by Lemma \ref{lemma:new-norm-semidiscrete-error} and \eqref{eq:quotient-scaling-estimate}
    \begin{align*}
        \Vert A_2^{d/4 -1/2 + 2\epsilon} A_{2,h}^{-\gamma} \pi_h (I - \tau_h) \Vert_{L(H)} &\leq \Vert A_2^{d/4 - 1/2 + 2\epsilon} (A_{2,h}^{-\gamma} \pi_h - A_2^{-\gamma} + A_2^{-\gamma}) \Vert_{L(H)} \Vert I - \tau_h \Vert_{L(H)} \\
        &\leq (\Vert A_2^{-\gamma + d / 4 - 1/2 + 2\epsilon} \Vert_{L(H)} \\
        &\qquad + \Vert A_2^{d/4 - 1/2 + 2\epsilon} (A_2^{-\gamma} - A_{2,h}^{-\gamma} \pi_h) \Vert_{L(H)}) C h^2.
    \end{align*}
    The first term above is bounded since $\gamma > d/4 - 1/2 + 2\epsilon$. The second term is bounded as shown above. Combining the estimates above finishes the estimate of $(ii)$ and the proof.
\end{proof}

We are now ready to prove Theorem \ref{theorem:strong-rate}.
\begin{proof}[Proof of Theorem \ref{theorem:strong-rate}]
Our approximate mild solution \eqref{eq:fully-discrete-scheme} can be extended from discrete times to all times $t \geq 0$ as
\begin{align}\label{eq:fully-discrete-2}
    u_{h,\Delta t}(t) := S_{h,\Delta t}(t) \pi_h \xi + \int_0^t S_{h,\Delta t}(t-s) Q_k^{-\gamma}(A_{2,h}) \pi_h \tau_h \, dW.
\end{align}
We decompose the error as follows
\begin{align*}
    u(t) - u_{h,\Delta t}(t) &= (S_1(t) - S_{h,\Delta t}(t) \pi_h) \xi \\
    &\qquad+ \bigg(\int_0^t S_1(t-s) A_2^{-\gamma} \, dW - \int_0^t S_{h,\Delta t}(t-s) A_{2,h}^{-\gamma} \pi_h \tau_h \, dW \bigg) \\
    &\qquad+ \bigg(\int_0^t S_{h,\Delta t}(t-s) A_{2,h}^{-\gamma} \pi_h \tau_h \, dW - \int_0^t S_{h,\Delta t}(t-s) Q_k^{-\gamma}(A_{2,h}) \pi_h \tau_h \, dW \bigg) \\
    &=: (i) + (ii) + (iii).
\end{align*}
We have,
\begin{align*}
    \Vert (i) \Vert_{L^p(\Omega ; H)} \leq C e^{c \lambda t} t^{-\theta/2 + \rho/2} (h^{\theta} + \Delta t^{\theta/2}) \Vert (\lambda + A_1)^{\rho / 2} \xi \Vert_{L^p(\Omega ; H)},
\end{align*}
for any $\theta \in [0,2]$ and $\rho \in [-1, \theta] \cap [-2 + \theta, \theta]$, by Lemma \ref{lemma:fully-discrete-semigroup-approximation}. For $(ii)$, we use Lemma \ref{lemma:error-stochastic-convolution}. 

Finally for $(iii)$, we have using the Burkholder--Davis--Gundy inequality (Theorem 4.36 in \cite{da-prato}), Lemma \ref{lemma:fully-discrete-semigroup-smoothing}, \ref{lemma:A_h-quadrature} and \ref{lemma:hilbert-schmidt-finite-rank}
\begin{align*}
    &\Vert \int_0^t S_{h,\Delta t}(t-s) (A_{2,h}^{-\gamma} - Q_k^{-\gamma}(A_{2,h})) \pi_h \tau_h \, dW \Vert_{L^p(\Omega; H)}^p \\
    &\qquad\leq C_p \bigg( \int_0^t \Vert S_{h,\Delta t}(t-s) (A_{2,h}^{-\gamma} - Q_k^{-\gamma}(A_{2,h})) \pi_h \tau_h \Vert_{L_2(H)}^2 \, ds \bigg)^{p/2} \\
    &\qquad\leq C_p \bigg( \int_0^t e^{2c(\lambda - \delta) s} \Vert (A_{2,h}^{-\gamma} - Q_k^{-\gamma}(A_{2,h})) \pi_h \tau_h \Vert_{L_2(H)}^2 \, ds \bigg)^{p/2} \\
    &\qquad\leq C_p e^{p c \lambda t} \bigg( \int_0^t e^{-2c \delta s} \Vert (A_{2,h}^{-\gamma} - Q_k^{-\gamma}(A_{2,h})) \Vert_{L(H)}^2 \Vert \pi_h \tau_h \Vert_{L_2(H)}^2 \, ds \bigg)^{p/2} \\
    &\qquad\leq C_p e^{p c \lambda t} (e^{-\frac{\pi^2}{2k}} N_h^{1/2})^p, 
\end{align*} 
Using the bound on $k$ in terms of $h$, the bound $N_h \leq C h^{-d}$ which follows from \ref{eq:quasi-uniform}, this term is bounded by $C (h^{2\gamma + 1 - d/2})^p$. If $\gamma = 1$, this term vanishes.

In total, we get,
\begin{align*}
    \Vert u(t) - u_{h,\Delta t}(t) \Vert_{L^p(\Omega ; H)} &\leq C_{p,\theta} e^{c \lambda t} (h^{\theta} + \Delta t^{\theta / 2}) (1 + t^{-\theta / 2 + \rho / 2} \Vert (\lambda + A_1)^{\rho / 2} \xi \Vert_{L^p(\Omega ; H)}),
\end{align*}
for any $\theta \in [0, 2\gamma + 1 - d/2) \cap [0,2]$, and $\rho \in [-1, \theta] \cap [-2 + \theta, \theta]$. 
\end{proof}

\section{Numerical examples and verification of convergence rate}

We numerically verify the convergence rates obtained in Theorem \ref{theorem:strong-rate} and Corollary \ref{cor:pathwise-rate} for the model,
\begin{align}\label{eq:example-model}
    du(t) = \Delta_{\Gamma} u(t) \, dt + (I - \Delta_{\Gamma})^{-\gamma} \, dW(t), \quad u(0) = 0,
\end{align}
where $\Delta_{\Gamma}$ is related to the bilinear form $\int_{\Gamma} \nabla_{\Gamma} u \cdot \nabla_{\Gamma} v \, d\sigma$, on $V = H^1(\Gamma)$. In Figure \ref{fig:rates}, we consider the domains $\Gamma = \mathbb{S}^1$ and $\Gamma = \mathbb{S}^2$. We approximate the relative pathwise error at time $t = 1$ by
\begin{align}\label{eq:relative-pathwise-error}
    e_{h,\Delta t} := \frac{\Vert u_{h,\Delta t}(1)-u_{\tilde{h},\widetilde{\Delta t}}(1) \Vert_H}{\Vert u_{\tilde{h},\widetilde{\Delta t}}(1)\Vert_H},
\end{align}
where a coarse approximation, $u_{h, \Delta t}(1)$, is compared to a reference solution, $u_{\tilde{h}, \widetilde{\Delta t}}(1)$, based on a finer space and time resolution, $\tilde{h}$ and $\widetilde{\Delta t}$. The quadrature resolution $k$ is fixed and equal to $0.5$ in all experiments.

To compute $u_{h, \Delta t}$ for different resolutions $h, \Delta t$, using the same Wiener process, $W$, we note that (see, e.g., Lemma \ref{lemma:auxiliary-basis-coef})
\begin{align*}
    \widetilde{\Delta t}^{1/2} M_{\tilde{h}}^{1/2} \varrho_{\tilde{h}}^n := \begin{pmatrix}
        (\tilde{\varphi}_1, \iota_{\tilde{h}}^{-1} \pi_{\tilde{h}} \tau_{\tilde{h}} (W(\tilde{t}_{n+1}) - W(\tilde{t}_n)))_{L^2(\Gamma_{\tilde{h}})} \\
        \vdots \\
        (\tilde{\varphi}_{N_{\tilde{h}}}, \iota_{\tilde{h}}^{-1} \pi_{\tilde{h}} \tau_{\tilde{h}} (W(\tilde{t}_{n+1}) - W(\tilde{t}_n)))_{L^2(\Gamma_{\tilde{h}})}
    \end{pmatrix},
\end{align*}
where $\tilde{t}_n = n \widetilde{\Delta t}$ and $\tilde{\varphi}_j$ $j = 1, \dots, N_{\tilde{h}}$ are the nodal basis functions based on the fine resolution mesh. Thus, for a suitable coarser time resolution, $\Delta t$ with $t_n := n \Delta t$, we can construct increments $(\tilde{\varphi}_j, \iota_{\tilde{h}}^{-1} \pi_{\tilde{h}} \tau_{\tilde{h}} (W(t_{n+1}) - W(t_n)))_{L^2(\Gamma_{\tilde{h}})}$ by summing increments of $\varrho_{\tilde{h}}^n$ based on the fine time resolution. Moreover, for a suitable coarser space resolution $h$, we can approximate $\varrho_h^n$ by
\begin{align*}
    M_h^{1/2} \varrho_h^n \approx A M_{\tilde{h}}^{1/2} \varrho_{\tilde{h}}^n,
\end{align*}
where $A_{ij} = \varphi_i(\tilde{x}_j)$, $\tilde{x}_j$ is the vertex corresponding to $\tilde{\varphi}_j$ projected down onto the coarser simplex where $\varphi_i$ is defined.

Approximations $\hat{e}_{h,\Delta t}$ of \eqref{eq:relative-pathwise-error} are computed as follows: we simulate a realization of the reference and coarse solution, $u_{\tilde{h},\widetilde{\Delta t}}(1) = \sum_{j = 1}^{N_{\tilde{h}}} \tilde{\alpha}_j \tilde{\varphi}_j$ and $u_{h,\Delta t}(1) = \sum_{j = 1}^{N_h}\alpha_j \varphi_j$, respectively. Then,
\begin{align}\label{eq:relative-pathwise-error-approxmiation}
    \hat{e}_{h,\Delta t}^2 := \frac{ (A^T\alpha - \tilde{\alpha})^T M_{\tilde{h}}(A^T\alpha-\tilde{\alpha})}{ \tilde{\alpha}^T M_{\tilde{h}} \tilde{\alpha}}.
\end{align}

\begin{example}\label{ex:1d-rates}
We consider \eqref{eq:example-model} with $\Gamma = \mathbb{S}^1$, $k = 0.5$, $\widetilde{\Delta t} = 2^{-20}$ and $\tilde{h} = 2^{-11}$. Coarser solutions are computed at space resolutions $h \in \{2^{-2}, \dots, 2^{-7}\}$ and time resolutions $\Delta t \in \{2^{-7}, \dots, 2^{-15}\}$. The first row in Figure \ref{fig:rates} shows the numerical convergence rates for $\gamma = 0$, $\gamma = 0.25$, $\gamma = 0.5$ and $\gamma = 0.75$, together with corresponding theoretical rates from Corollary \ref{cor:pathwise-rate}.
\end{example}

\begin{example}\label{ex:2d-rates}
We consider \eqref{eq:example-model} with $\Gamma = \mathbb{S}^2$, $k = 0.5$, $\widetilde{\Delta t} = 2^{-15}$ and $\tilde{h} = 2^{-5.5}$. Coarser solutions are computed at space resolutions $h \in \{2^{-0.5}, \dots, 2^{-3.5}\}$, and time resolutions $\Delta t \in \{2^{-5}, \dots, 2^{-9}\}$. The second row in Figure \ref{fig:rates} shows the numerical convergence rates for $\gamma = 0.25$, $\gamma = 0.5$, $\gamma = 0.75$ and $\gamma = 1$, together with corresponding theoretical rates from Theorem \ref{theorem:strong-rate}.
\end{example}

In both Example \ref{ex:1d-rates} and \ref{ex:2d-rates} we speed up computations by taking advantage of the fact that $(I-\Delta_{\Gamma})^{-\gamma}$ and $\Delta_{\Gamma}$ commute so that $(I-\Delta_{\Gamma})^{-\gamma}$ only needs to be applied at the final time.

\begin{figure}
    \centering
    \subcaptionbox{Dashed lines show rates $\frac{1}{2}, 1, \frac{3}{2}, 2$.}{
        \includegraphics[width = 6.5cm]{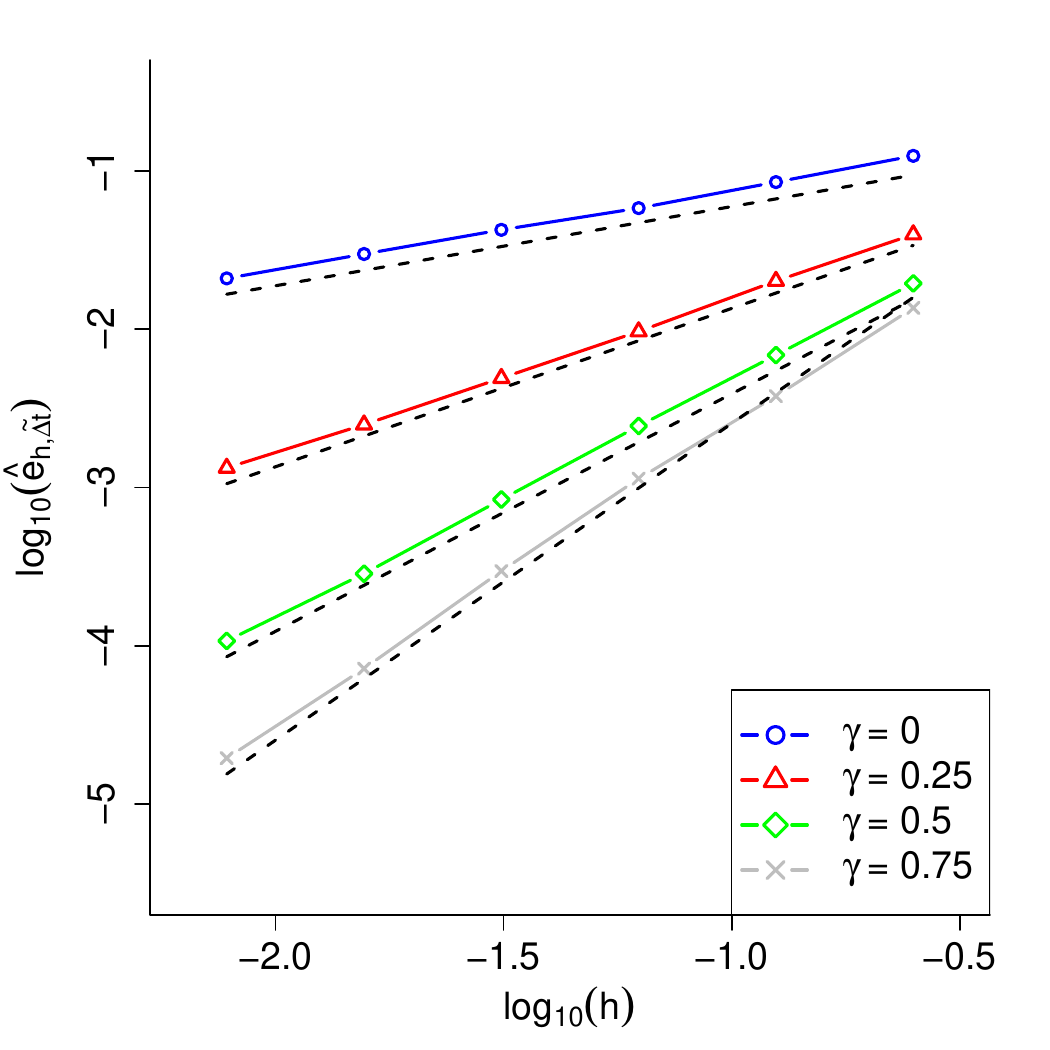}
    }
    \subcaptionbox{Dashed lines show rates $\frac{1}{4},\frac{1}{2},\frac{3}{4}, 1$}{
        \includegraphics[width = 6.5cm]{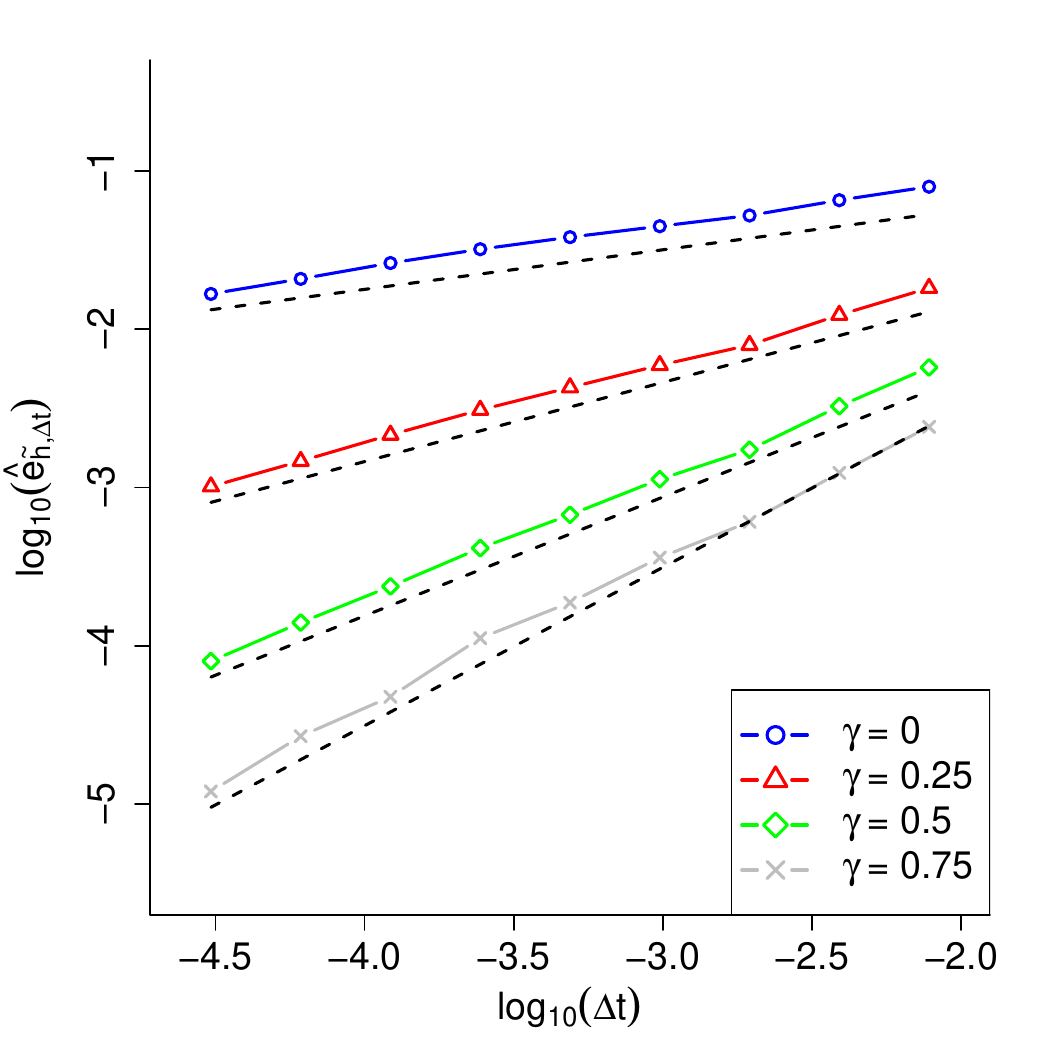}
    } \\
    \subcaptionbox{Dashed lines show rates $\frac{1}{2},1,\frac{3}{2},2$}{
        \includegraphics[width = 6.5cm]{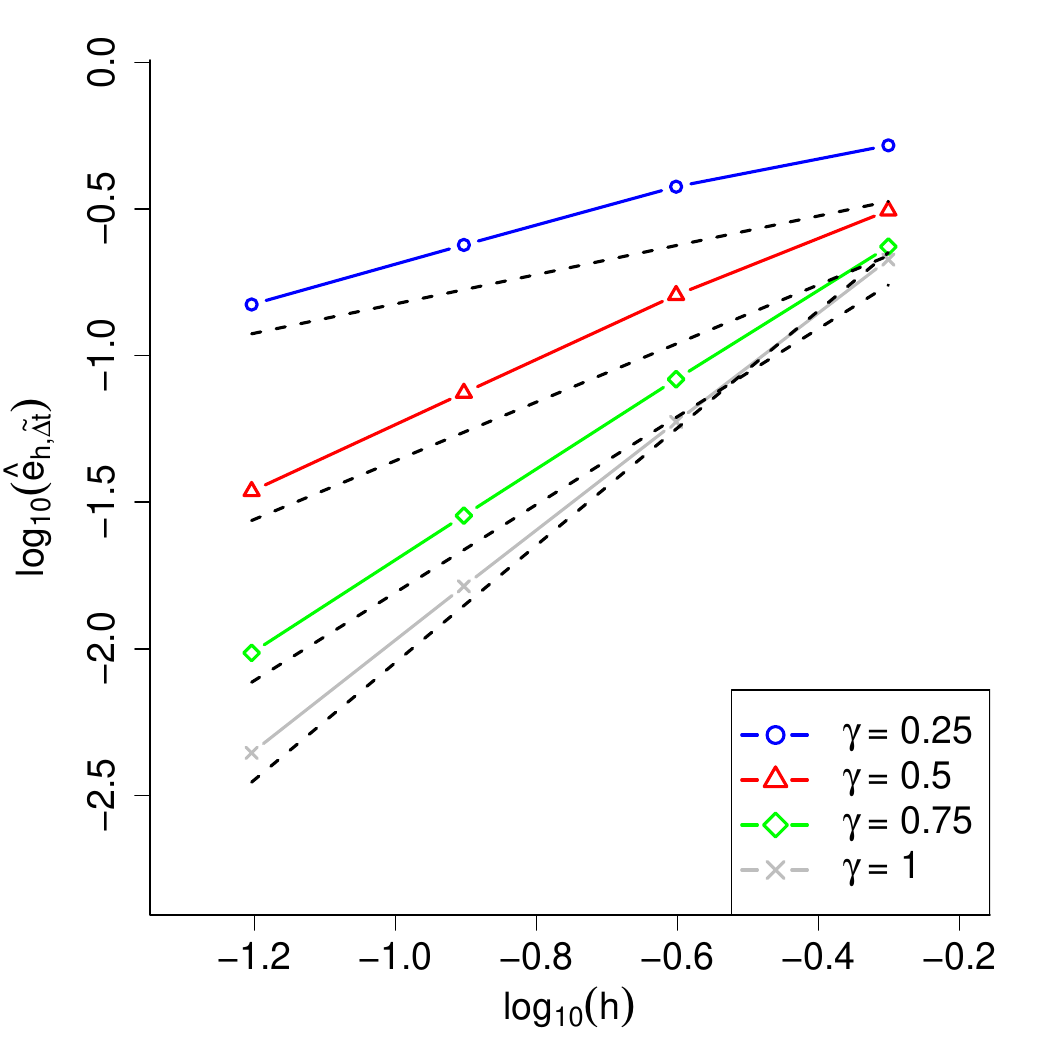}
    }
    \subcaptionbox{Dashed lines show rates $\frac{1}{4},\frac{1}{2},\frac{3}{4},1$.}{
        \includegraphics[width = 6.5cm]{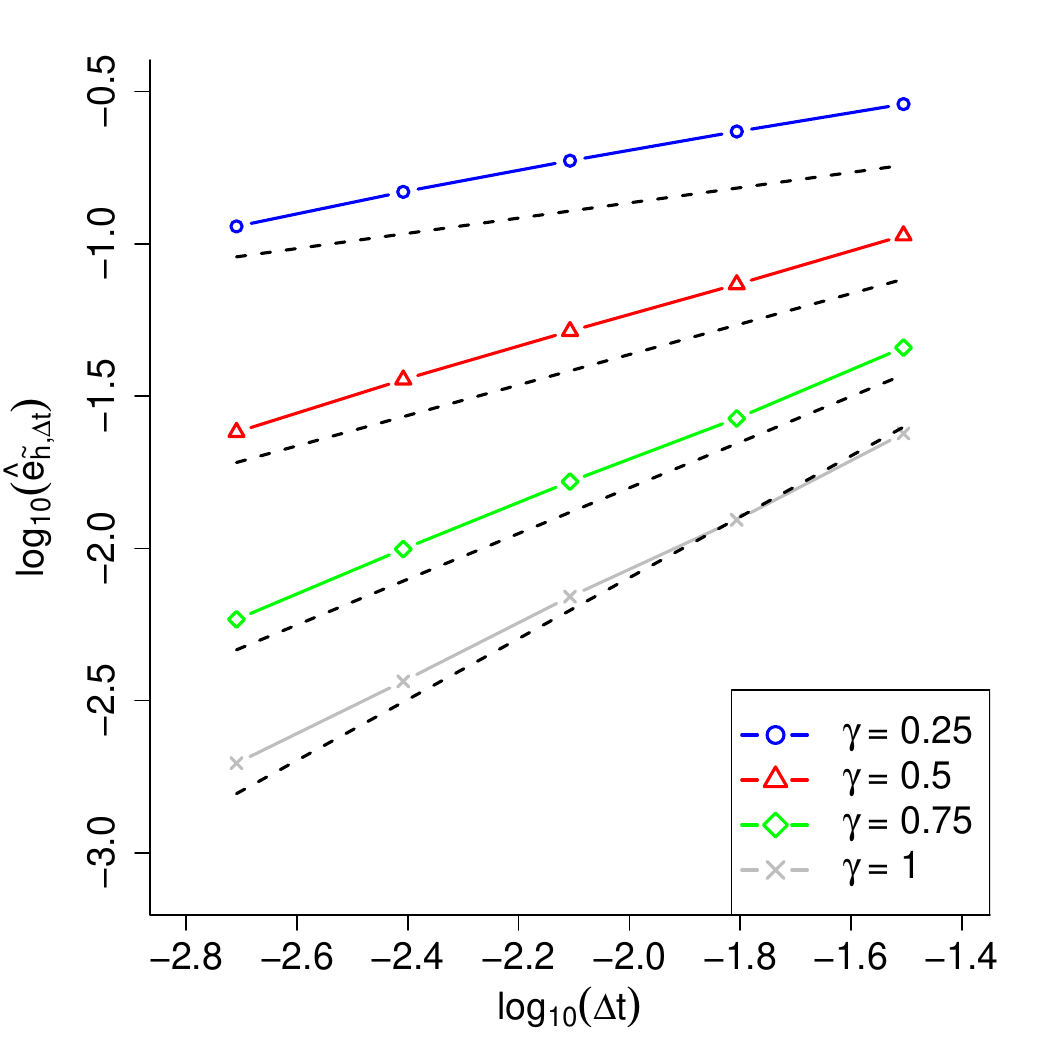}
    }
    \caption{Relative errors. Row 1 and 2 corresponds to Example 1 and 2, while Column 1 and 2 show rates in space and time, respectively. The dashed lines show corresponding theoretical rates. \label{fig:rates}} 
\end{figure}

\begin{example}\label{ex:motivation}
In our final example we consider two models with spatially varying coefficients. For both models, the parameters are $\alpha_2 = 0$, $A_2 = I$, $b_2 = 0$, while $\alpha_1 = 1$, $\mathcal{A}_2 = I + 5 v v^T$ with
\begin{align*}
    v = \cos^2(\pi x_3 / 2) (I - \nu \nu^T) (x_2, -x_1, 0), \quad \text{and} \quad b_1 := -(I - \nu \nu^T) (0, 0, x_3 / 2),
\end{align*}
where $\nu(x)$ is the unit surface normal of $\Gamma$ at $x$. We use the discrete coefficients $\mathcal{A}_{j,h} = (I - \nu_h \nu_h^T) \mathcal{A}_j$, $b_{j,h} = (I - \nu_h \nu_h^T) b_j$, evaluated at $\Gamma_h$, with $\nu_h$ being the unit normal for the discrete surface. (It is worth noting that we have not verified all conditions of Assumption \ref{assumption:model} and \ref{assumption:fem} for this particular choice of coefficients, but we still include it as an example). 

We consider $\Gamma_1 := \mathbb{S}^2$ and $\Gamma_2 := f (\mathbb{S}^2)$, with $f(x_1, x_2, x_3) := (1 - 0.5 \cos^2(\pi x_3)) (x_1, x_2, 0) + (0,0,x_3)$. In the first row of Figure \ref{fig:motivation} a realization of $u(1)$ using domain $\Gamma_1$ is shown for different values of $\gamma$. In the second row of Figure \ref{fig:motivation} a realization of $u(1)$ using domain $\Gamma_2$ is shown for the same values of $\gamma$. 
\end{example}

\subsection*{Acknowledgements}

The research of ØSA and GAF was supported by the project IMod---Partial differential equations, statistics and data: an interdisciplinary approach to data-based modelling, project number 325114, from the Research Council of Norway. AL's work was supported in part by the Swedish Research Council (VR) through grant no.\ 2020-04170, by the Wallenberg AI, Autonomous Systems and Software Program (WASP) funded by the Knut and Alice Wallenberg Foundation, by the Chalmers AI Research Centre (CHAIR), and by the European Union (ERC, StochMan, 101088589). Views and opinions expressed are however those of the author(s) only and do not necessarily reflect those of the European Union or the European Research Council Executive Agency. Neither the European Union nor the granting authority can be held responsible for them.

\addcontentsline{toc}{section}{References}
\bibliographystyle{abbrv}
\bibliography{refs.bib}

\appendix
\section{Derivation of (\ref{eq:scheme-basis-coefficients})}\label{app:derivation-of-system}

In order to rewrite \eqref{eq:fully-discrete-scheme} as a system of equations for the coefficients of $\iota_h^{-1} u_{h,\Delta t}$ in the nodal basis, the following two lemmas are key. 
\begin{lemma}\label{lemma:hilbert-schmidt-finite-rank}
    Let $B_h \in L(H, \iota_h V_h)$. Then,
    \begin{align*}
        \Vert B_h \Vert_{L_2(H)} \leq \Vert B_h \Vert_{L(H)} N_h^{1/2}.
    \end{align*}
\end{lemma}
\begin{proof}
    Let $\{ e_j \}_{j=1}^{\infty}$ be an $H$-orthonormal basis of $H$, and let $\{ e_{j,h} \}_{j=1}^{N_h}$ be an $H$-orthonormal basis of $\iota_h V_h$. We have,
    \begin{align*}
        \Vert B_h \Vert_{L_2(H)}^2 &= \sum_{j=1}^{\infty} (B_h e_j, B_h e_j)_H = \sum_{j=1}^{\infty} \sum_{k = 1}^{N_h} (B_h e_j, e_{k,h})_H (B_h e_j, e_{k,h})_H = \sum_{k=1}^{N_h} \Vert B_h^* e_{k,h} \Vert_H^2,
    \end{align*}
    and using that $\Vert B_h^* \Vert_{L(H)} = \Vert B_h \Vert_{L(H)}$ gives the estimate. 
\end{proof}

\begin{lemma}\label{lemma:discrete-cylindrical-wiener-process}
    Let $W$ be a cylindrical Wiener process on $H$. Then $\iota_h^{-1} \pi_h \tau_h W$ is a cylindrical Wiener process on $V_h$, in the sense that,
    \begin{align*}
        \iota_h^{-1} \pi_h \tau_h W (t) = \sum_{j = 1}^{N_h} \beta_j(t) e_{j,h}, \quad P\text{-a.s.}
    \end{align*}
    where $\beta_j$ are independent scalar Brownian motions, and $\{ e_{j,h} \}_{j = 1}^{N_h}$ is an $L^2(\Gamma_h)$-orthonormal basis of $V_h$. 
\end{lemma}
\begin{proof}
    We have that $\iota_h^{-1} \pi_h \tau_h \in L_2(H, V_h)$ (where $V_h$ has the $L^2(\Gamma_h)$-norm) since it has finite rank. It follows that $\iota_h^{-1} \pi_h \tau_h W(t) \in V_h$ $P$-a.s. (see, e.g., Section 4.2.1 in \cite{da-prato}). Therefore, for any $L^2(\Gamma_h)$-orthonormal basis $\{ e_{j,h} \}_{j = 1}^{N_h}$ of $V_h$ we must have
    \begin{align*}
        \iota_h^{-1} \pi_h \tau_h W(t) = \sum_{j = 1}^{N_h} (\iota_h^{-1} \pi_h \tau_h W(t), e_{j,h})_{L^2(\Gamma_h)} e_{j,h}, \quad P\text{-a.s.}
    \end{align*}
    To check that the law of $(\iota_h^{-1} \pi_h \tau_h W(t), e_{j,h})_{L^2(\Gamma_h)}$ are those of independent scalar Brownian motions for each $j$, we note that
    \begin{align*}
        (\iota_h^{-1} \pi_h \tau_h W(t), e_{j,h})_{L^2(\Gamma_h)} = (\iota_h^{-1} \tau_h W(t), e_{j,h})_{L^2(\Gamma_h)} = (W(t), (\iota_h \delta_h^{-1/2}) \iota_h e_{j,h})_H.
    \end{align*}
    Now the result follows since $\{ (\iota_h \delta_h^{-1/2}) \iota_h e_{j,h} \}_{j = 1}^{N_h}$ defines an $H$-orthonormal basis of $\iota_h V_h$.
\end{proof}

To rewrite \eqref{eq:fully-discrete-scheme} as an equation involving the basis coefficients of $\iota_h^{-1} u_{h,\Delta t}$ in the nodal basis of $V_h$, we start by applying $\iota_h^{-1}$ to the equation. Solving the modified equation, with left and right hand side in $V_h$, is the same as solving the system of equations, 
\begin{align}\label{eq:system-of-equations-appendix}
\begin{split}
    &(\iota_h^{-1} (I + \Delta t A_{1,h})u_{h,\Delta t}(t_{n+1}) , \varphi_j)_{L^2(\Gamma_h)} \\
    &\qquad = ( \iota_h^{-1} u_{h,\Delta t}(t_n) + \iota_h^{-1} Q_k^{-\gamma}(A_{2,h}) \pi_h \tau_h (W(t_{n+1}) - W(t_n)), \varphi_j)_{L^2(\Gamma_h)},
\end{split}
\end{align}
$j = 1, \dots, N_h$, with $\{\varphi_j\}_{j = 1}^{N_h}$ the nodal basis of $V_h$. The following lemma is useful for setting up the system above.
\begin{lemma}\label{lemma:auxiliary-basis-coef}
    We have
    \begin{align*}
        \iota_h^{-1} Q_k^{-\gamma}(A_{2,h}) \pi_h \tau_h (W(t_{n+1}) - W(t_n)) = \sum_{j=1}^{N_h} \theta_j \varphi_j, \quad P\text{-a.s.}
    \end{align*}
    where 
    \begin{align*}
        \theta = 
        \begin{cases}
            \Delta t^{1/2} \frac{k \sin(\pi \gamma)}{\pi} \sum_{l = -M}^N e^{(1-\gamma) y_l} (e^{y_l} M_h + K_h)^{-1} M_h^{1/2} \varrho_h^n, \quad &\gamma \in (0,1), \\
            \Delta t^{1/2} K_h^{-1} M_h^{1/2} \varrho_h^n, \quad &\gamma = 1,
        \end{cases}
    \end{align*}
    the matrices $M_h$ and $K_h$ are given in \eqref{eq:fem-matrices}, and
    \begin{align*}
        \varrho_h^n := \Delta t^{-1/2} M_h^{-1 / 2}
        \begin{pmatrix}
            &(\iota_h^{-1} \pi_h \tau_h (W(t_{n+1}) - W(t_n)), \varphi_1)_{L^2(\Gamma_h)} \\
            &\vdots \\
            &(\iota_h^{-1} \pi_h \tau_h (W(t_{n+1}) - W(t_n)), \varphi_{N_h})_{L^2(\Gamma_h)}
        \end{pmatrix}
        \sim \mathcal{N}(0,I)
    \end{align*}
    are $N_h$-dimensional multivariate Gaussian. 
\end{lemma}
\begin{proof}
    Set for ease of notation,
    \begin{align*}
        f := \iota_h^{-1} \pi_h \tau_h (W(t_{n+1}) - W(t_n)).
    \end{align*}
    By Lemma \ref{lemma:discrete-cylindrical-wiener-process}, $f$ is an $L^2(\Gamma_h)$-valued Gaussian random variable with covariance operator $\Delta t \mathcal{P}_h$. Recall $A_{j,h} := \iota_h A_{j,h}' \iota_h^{-1}$, and note that 
    \begin{align*}
        \iota_h^{-1} Q_k^{-\gamma}(A_{2,h}) \pi_h \tau_h (W(t_{n+1}) - W(t_n)) = Q_k^{-\gamma}(A_{2,h}') f.
    \end{align*}

    For the expression for $\theta$ in the case of $\gamma = 1$, note that $(A_{2,h}')^{-1} f$ is the solution $g \in V_h$ of
    \begin{align*}
        A_{2,h}' g = f.
    \end{align*}
    Since $f \in V_h$ $P$-a.s. by Lemma \ref{lemma:discrete-cylindrical-wiener-process}, solving the equation above is the same as solving the system of equations,
    \begin{align}\label{eq:system-1}
        \begin{pmatrix}
            (A_{2,h}' g, \varphi_1)_{L^2(\Gamma_h)} \\
            \vdots \\
            (A_{2,h}' g, \varphi_{N_h})_{L^2(\Gamma_h)}
        \end{pmatrix}
        =
        \begin{pmatrix}
            (f, \varphi_1)_{L^2(\Gamma_h)} \\
            \vdots \\
            (f, \varphi_{N_h})_{L^2(\Gamma_h)}
        \end{pmatrix}.
    \end{align}
    By Lemma \ref{lemma:discrete-cylindrical-wiener-process} we have,
    \begin{align*}
        E[(f, \varphi_i)_{L^2(\Gamma_h)} (f, \varphi_j)_{L^2(\Gamma_h)}] = \Delta t (\varphi_i, \varphi_j)_{L^2(\Gamma_h)},
    \end{align*}
    and so the covariance matrix of $((f, \varphi_1)_{L^2(\Gamma_h)}, \dots, (f, \varphi_{N_h})_{L^2(\Gamma_h)})^T$ is the (scaled) mass matrix, $\Delta t M_h$. It follows that 
    \begin{align*}
        ((f, \varphi_1)_{L^2(\Gamma_h)}, \dots, (f, \varphi_{N_h})_{L^2(\Gamma_h)})^T = \Delta t^{1/2} M_h^{1/2} \varrho_h^n,
    \end{align*}
    $P$-a.s., with $\varrho_h^n$ as above. For the left hand side of \eqref{eq:system-1} we insert $g = \sum_{j = 1}^{N_h} \theta_j \varphi_j$, and find, using that $(A_{2,h}' \varphi_j, \varphi_i)_{L^2(\Gamma_h)} = a_{2,h}(\varphi_j, \varphi_i)$,
    \begin{align*}
        \begin{pmatrix}
            (A_{2,h}' g, \varphi_1)_{L^2(\Gamma_h)} \\
            \vdots \\
            (A_{2,h}' g, \varphi_{N_h})_{L^2(\Gamma_h)}
        \end{pmatrix}
        = K_h
        \begin{pmatrix}
            \theta_1 \\
            \vdots \\
            \theta_{N_h}
        \end{pmatrix},
    \end{align*}
    where $K_h$ is as in \eqref{eq:fem-matrices}. Combining these observations, we see that \eqref{eq:system-1} has solution $(g_1, \dots, g_{N_h})$ given by,
    \begin{align*}
        (g_1, \dots, g_{N_h})^T = \Delta t^{1/2} K_h^{-1} M_h^{1/2} \varrho_h^n.
    \end{align*}

    The identity for $\theta$ when $\gamma \in (0,1)$ follows similarly. Note that,
    \begin{align*}
        Q_k^{-\gamma}(A_{2,h}') f = \frac{k \sin(\pi \gamma)}{\pi} \sum_{j = -M}^N e^{(1-\gamma) y_j} g^{(j)}, 
    \end{align*}
    where $g^{(j)} \in V_h$, is the solution of the equation, 
    \begin{align*}
        (e^{y_j} I + A_{2,h}') g^{(j)} = f, \quad j = -M, \dots, N.
    \end{align*}
    As for the previous term, to solve this equation we insert $g^{(j)} = \sum_{l=1}^{N_h} \theta^{(j)}_l \varphi_l$ into the equation and integrate against the nodal basis, to find,
    \begin{align*}
        (e^{y_j} M_h + K_h) 
        \begin{pmatrix}
            \theta_1^{(j)} \\
            \vdots \\
            \theta_{N_h}^{(j)}
        \end{pmatrix}
        = \Delta t^{1/2} M_h^{1/2} \varrho_h^n.
    \end{align*}
    Summing up the vectors $(\theta^{(j)}_1, \dots, \theta^{(j)}_{N_h}), \ j = -M, \dots, N$ we find the coeffcients, $\theta$, of $Q_k^{-\gamma}(A_{2,h}') f$ in the nodal basis.
\end{proof}

Now we can insert the identities of Lemma \ref{lemma:auxiliary-basis-coef} into the system of equations \eqref{eq:system-of-equations-appendix} to arrive at \eqref{eq:scheme-basis-coefficients}. Note that for any $g = \sum_{j=1}^{N_h} \theta_j \varphi_j$, we have,
\begin{align*}
    \begin{pmatrix}
        (g, \varphi_1)_{L^2(\Gamma_h)} \\
        \vdots \\
        (g, \varphi_{N_h})_{L^2(\Gamma_h)}
    \end{pmatrix}
    = M_h \theta,
\end{align*} 
and so,
\begin{align*}
    \begin{pmatrix}
        (\iota_h^{-1} (I + \Delta t A_{1,h})u_{h,\Delta t}(t_{n+1}), \varphi_1)_{L^2(\Gamma_h)} \\
        \vdots \\
        (\iota_h^{-1} (I + \Delta t A_{1,h})u_{h,\Delta t}(t_{n+1}), \varphi_{N_h})_{L^2(\Gamma_h)}
    \end{pmatrix}
    = (M_h + \Delta t T_{h}) \alpha^{n+1},
\end{align*} 
while, 
\begin{align*}
    &
    \begin{pmatrix}
        (\iota_h^{-1} Q_k^{-\gamma}(A_{2,h}) \pi_h (W(t_{n+1}) - W(t_n)), \varphi_1)_{L^2(\Gamma_h)} \\
        \vdots \\
        (\iota_h^{-1} Q_k^{-\gamma}(A_{2,h}) \pi_h (W(t_{n+1}) - W(t_n)), \varphi_{N_h})_{L^2(\Gamma_h)}
    \end{pmatrix}
    \\
    &\qquad= 
    \begin{cases}
        M_h \Delta t^{1/2} \frac{k \sin(\pi \gamma)}{\pi} \sum_{j = -M}^N e^{(1-\gamma) y_j} (e^{y_j} M_h + K_h)^{-1} M_h^{1/2} \varrho_h^n, \quad &\gamma \in (0,1), \\
        M_h \Delta t^{1/2} K_{h}^{-1} M_h^{1/2} \varrho_h^n, \quad &\gamma = 1.
    \end{cases}
\end{align*}

\end{document}